\documentclass[10pt,letterpaper]{amsart} 

\usepackage{color}
\usepackage{multicol}
\usepackage[letterpaper, margin=1in]{geometry}
\usepackage{subfig}
\usepackage[overload]{empheq}

\usepackage{listings}
\usepackage{amsmath, amsthm, amssymb, wasysym, verbatim, bbm, color, graphics}
\usepackage{enumerate}
\usepackage{url}
\usepackage{mathtools}
\usepackage[colorlinks,linkcolor=blue]{hyperref}
\usepackage{romannum}
\usepackage{xfrac}
\usepackage{float}
\usepackage{amsmath,bm}
\usepackage[title]{appendix}
\usepackage  {todonotes}

\newcommand{\R}{\mathbb{R}}

\DeclarePairedDelimiter{\floor}{\lfloor}{\rfloor}

\DeclareMathOperator{\tr}{\text{tr}}

\newcommand\norm[1]{\left\lVert#1\right\rVert}
\newcommand{\Grad}{\nabla}
\newcommand{\Div}{\nabla\cdot}
\newcommand{\dom}{\Omega}

\newcommand{\uint}{\Tilde{\pmb{u}}^{n+1}}
\newcommand{\uold}{\pmb{u}^n}
\newcommand{\unew}{\pmb{u}^{n+1}}
\newcommand{\uN}{\pmb{u}^{N+1}}
\newcommand{\Dtuone}{\frac{\uint-\uold}{\Delta t}}
\newcommand{\Dtutwo}{\frac{\unew-\uint}{\Delta t}}

\newcommand{\Qold}{\pmb{Q}^{n}}
\newcommand{\Qnew}{\pmb{Q}^{n+1}}
\newcommand{\QN}{\pmb{Q}^{N+1}}
\newcommand{\Pold}{\pmb{P}^{n}}
\newcommand{\Snew}{\pmb{s}^{n+1}}
\newcommand{\sigmanew} {\pmb{\Sigma}^{n+1}}
\newcommand{\Hnew}{\pmb{H}^{n+1}}
\newcommand{\Hold}{\pmb{H}^n}
\newcommand{\usol}{\pmb{u}_{\Delta t}}
\newcommand{\usolsubseq}{\pmb{u}_{\Delta t_m}}
\newcommand{\Wsol}{\pmb{W}_{\Delta t}}
\newcommand{\Dsol}{\pmb{D}_{\Delta t}}

\newcommand{\Dsolsubseq}{\pmb{D}_{\Delta t_m}}
\newcommand{\uappr}{\pmb{u}^*_{\Delta t}}
\newcommand{\uapprsubseq}{\pmb{u}^*_{\Delta t_m}}
\newcommand{\Qsol}{\pmb{Q}_{\Delta t}}
\newcommand{\Qini}{\pmb{Q}_{\text{in}}}
\newcommand{\Qsolsubseq}{\pmb{Q}_{\Delta t_m}}
\newcommand{\Qbacksol}{\pmb{Q}^*_{\Delta t}}
\newcommand{\Qbacksolsubseq}{\pmb{Q}^*_{\Delta t_m}}
\newcommand{\Psol}{\pmb{P}_{\Delta t}}
\newcommand{\Psolsubseq}{\pmb{P}_{\Delta t_m}}
\newcommand{\rsol}{{r}_{\Delta t}}
\newcommand{\rsolsubseq}{{r}_{\Delta t_m}}
\newcommand{\Hsol}{\pmb{H}_{\Delta t}}
\newcommand{\Hsolsubseq}{\pmb{H}_{\Delta t_m}}
\newcommand{\Ssol}{\pmb{s}_{\Delta t}}
\newcommand{\Ssolsubseq}{\pmb{s}_{\Delta t_m}}
\newcommand{\Sigmasol}{\pmb{\Sigma}_{\Delta t}}
\newcommand{\Sigmasolsubseq}{\pmb{\Sigma}_{\Delta t_m}}
\newcommand{\testvar}{\pmb{\varphi}}
\newcommand{\testphi}{\pmb{\phi}}
\newcommand{\testpsi}{\pmb{\psi}}
\newcommand{\Qint}{\Tilde{\pmb{Q}}^{n+1}}

\newcommand{\rint}{\Tilde{r}^{n+1}}
\newcommand{\Wint}{\Tilde{\pmb{W}}^{n+1}}
\newcommand{\Dint}{\Tilde{\pmb{D}}^{n+1}}
\newcommand{\Qnewone}{Q^{n+1}_1}
\newcommand{\Qnewtwo}{Q^{n+1}_2}
\newcommand{\Qoldone}{Q^{n}_1}
\newcommand{\Qoldtwo}{Q^{n}_2}
\newcommand{\Hnewone}{H^{n+1}_1}
\newcommand{\Hnewtwo}{H^{n+1}_2}
\newcommand{\uintone}{\Tilde{u}^{n+1}_1}
\newcommand{\uinttwo}{\Tilde{u}^{n+1}_2}
\newcommand{\Wintmatrix}{\begin{pmatrix}
    0 & \frac{1}{2}(\partial_2\uintone-\partial_1\uinttwo)\\\frac{1}{2}(\partial_1\uinttwo-\partial_2\uintone) &0
\end{pmatrix}}
\newcommand{\Dintmatrix}{\begin{pmatrix}
    \partial_1\uintone & \frac{1}{2}(\partial_2\uintone+\partial_1\uinttwo)\\ \frac{1}{2}(\partial_2\uintone+\partial_1\uinttwo) & \partial_2\uinttwo
\end{pmatrix}}
\newcommand{\Qoldmatrix}{\begin{pmatrix}
    \Qoldone & \Qoldtwo\\
    \Qoldtwo & -\Qoldone
\end{pmatrix}}

\newcommand{\Snewone}{s^{n+1}_1}
\newcommand{\Snewtwo}{s^{n+1}_2}
\newcommand{\Hnewmatrix}{\begin{pmatrix} \Hnewone &\Hnewtwo\\
    \Hnewtwo &-\Hnewone
\end{pmatrix}}
\newcommand{\Sigmanewone}{\Sigma^{n+1}_1}
\newcommand{\Sigmanewtwo}{\Sigma^{n+1}_2}
\newcommand{\uintmatrix}{\begin{pmatrix}
    \uintone\\ \uinttwo
\end{pmatrix}}
\newcommand{\uoldmatrix}{\begin{pmatrix}
    u^n_1\\ u^n_2
\end{pmatrix}}
\newcommand{\HnewonetoQnewone}{\left(L\Delta-2(P_1^n)^2\right)\,\Qnewone}
\newcommand{\HnewtwotoQnewtwo}{\left(L\Delta-2(P_2^n)^2\right)\,\Qnewtwo}
\newcommand{\HnewonetoQnewtwo}{(-2P^n_1P^n_2)\,\Qnewtwo}
\newcommand{\HnewtwotoQnewone}{(-2P^n_1P^n_2)\,\Qnewone}

\newtheorem{lemma}{Lemma}[section] 

\newtheorem{corollary}[lemma]{Corollary}
\newtheorem{theorem}[lemma]{Theorem}
\newtheorem{remark}[lemma]{Remark}

\newtheorem*{maintheorem*}{Main Theorem}
\theoremstyle{definition}{\newtheorem{definition}[lemma]{Definition}}

\newcommand{\todofw}[1]{\todo[size=\tiny,author=FW,backgroundcolor=blue!20!white]{#1}}
\newcommand{\todosy}[1]{\todo[size=\tiny,author=Seven,backgroundcolor=red!20!white]{#1}}

\allowdisplaybreaks

\numberwithin{equation}{section}

\title[Analysis for Q-tensor flow]{On the Convergence of an IEQ-based first-order Numerical Scheme for the Beris-Edwards System}
\date{\today}
\thanks{The authors of this work gratefully acknowledge support by NSF grants DMS 1912854 and OIA-DMR 2021019.}

\author[F. Weber]{Franziska Weber}
\address[Franziska Weber]{\newline Department of Mathematics \newline University of California, Berkeley \newline Berkeley, CA 94720, USA.}
\email[]{fweber@math.berkeley.edu}
\author[Y. Yue]{Yukun Yue}
\address[Yukun Yue]{\newline Department of Mathematical Sciences \newline Carnegie Mellon University \newline 5000 Forbes Avenue, Pittsburgh, PA 15213, USA.}
\email[]{yukuny@andrew.cmu.edu}

\begin{document}
 \pagenumbering{arabic}
\maketitle
\begin{abstract}
We present a convergence analysis of an unconditionally energy-stable first-order semi-discrete numerical scheme designed for a hydrodynamic Q-tensor model, the so-called Beris-Edwards system, based on the Invariant Energy Quadratization Method (IEQ). The model consists of the Navier-Stokes equations for the fluid flow, coupled to the Q-tensor gradient flow describing the liquid crystal molecule alignment. By using the Invariant Energy Quadratization Method, we obtain a linearly implicit scheme, accelerating the computational speed. However, this introduces an auxiliary variable to replace the bulk potential energy and it is a priori unclear whether the reformulated system is equivalent to the Beris-Edward system. In this work, we prove stability properties of the scheme and show its convergence to a weak solution of the coupled liquid crystal system. We also demonstrate the equivalence of the reformulated and original systems in the weak sense. 
\end{abstract}

\section{Introduction}
Liquid crystal is an intermediate state of matter between the solid and liquid phase and usually exists in a specific temperature range. On one hand, it possesses the ability to flow of liquids, and on the other hand, the molecules are ordered, (neighboring molecules roughly point in the same direction) similar as in a classical solid. Due to this, liquid crystals have unique physical properties that are used in various real-life applications, such as monitors, screens, clocks, navigation systems, and others. Typically, liquid crystals consist of elongated molecules of identical size which can be pictured as rods. The inter-molecular forces make them align along a common axis~\cite{Sonnet2012,ANDRIENKO2018520}.

Mathematical models for the dynamics of liquid crystals have been intensively studied in the last decades. For an overview, see~\cite{Lin_theory2001, RevModPhys.46.617,Ericksen1962HydrostaticTO, Leslie1966, Leslie1968SomeCE} and the references therein. Here we will consider the Q-tensor model by Landau and de Gennes~\cite{de1993physics} and its numerical approximation. In this model, the orientation of the liquid crystal molecules is described by the Q-tensor, a symmetric and trace-free $d\times d$-matrix field where $d=2,3$ is the spatial dimension. It can be interpreted as the deviation of the second moment of the probability density of the directions of liquid crystal molecules from the isotropic state~\cite{majumdar_2010}. When the liquid crystal is in an equilibrium, the Q-tensor minimizes a free energy, the so-called Landau-de Gennes free energy~\cite{Majumdar_LandauDeGennes, Ball2017},
\begin{equation*}
\label{eq:Landau_DeGennes_energy}
E_{LG}(\pmb{Q})=\int_{\dom}\mathcal{F}_B(\pmb{Q})+\mathcal{F}_E(\pmb{Q}),
\end{equation*}
where $\dom\subset\mathbb{R}^d$, is the spatial domain, and we assume that it has a sufficiently smooth boundary. $\mathcal{F}_B$ is the bulk potential and $\mathcal{F}_E$ is the elastic energy density given by
\begin{equation*}
\label{eq:bulk_elastic}
\mathcal{F}_B(\pmb{Q})=\frac{a}{2}\tr(\pmb{Q}^2)-\frac{b}{3}\tr(\pmb{Q}^3)+\frac{c}{4}\left(\tr(\pmb{Q}^2)\right)^2, \quad\mathcal{F}_E(\pmb{Q})=\frac{L}{2}\lvert \nabla\pmb{Q}\rvert^2,
\end{equation*}
where $a, b, c, L$ are constants with $c, L>0$. In particular, $c>0$ will guarantee the existence of a lower bound of the bulk potential, which is vital for the following analysis. In a non-equilibrium situation, the dynamics of the Q-tensor are governed by a nonlinear system of PDEs, consisting of the gradient flow for the Q-tensor field coupled to the Navier-Stokes equations for the underlying fluid flow~\cite{beris1994thermodynamics, ZhaoWang_convexsplitting, ZHAO2017803},
\begin{subequations}\label{eq:pde_system}
  \begin{empheq}[left = \empheqlbrace]{align}    &\pmb{u}_t+(\pmb{u}\cdot\nabla)\pmb{u}=-\nabla p+\mu\Delta \pmb{u}+\nabla\cdot\pmb{\Sigma}-\pmb{H}\nabla\pmb{Q}\label{eq:original_system_ut}, \\  
& \Div \pmb{u} = 0\label{eq:original_system_divergence_free},\\
&\pmb{Q}_t+\pmb{u}\cdot\nabla\pmb{Q}-\pmb{S}=M\pmb{H},\label{eq:original_system_Qt}
\end{empheq}  
\end{subequations}
subject to initial and boundary conditions,
\begin{subequations}
\label{eq:bd_condition}
    \begin{empheq}
[left = \empheqlbrace]{align}
\left.\pmb{Q}\right|_{t=0}=\pmb{Q}_0,\qquad &\left.\pmb{Q}\right|_{\partial\Omega\times[0,T]}=0,\label{eq:Q_ib_condition}\\
\left.\pmb{u}\right|_{t=0}=\pmb{u}_0,\qquad &\left.\pmb{u}\right|_{\partial_\Omega\times[0,T]}=0\label{eq:u_ib_condition},
\end{empheq}
\end{subequations}
where $(\pmb{H}\nabla\pmb{Q})_k=\sum_{i,j=1}^d{H}_{ij}\partial_k{Q}_{ij}$ and $(\pmb{u}\cdot\nabla\pmb{Q})_{ij}=\sum_{k=1}^d u_k\partial_k Q_{ij}$ for all $1\leq k, i, j\leq d$. $\pmb{u}$ denotes the velocity field, and $p$ represents the pressure. The tensors $\pmb{S}$ and $\pmb{\Sigma}$ appearing in the system~\eqref{eq:original_system_ut}--\eqref{eq:original_system_Qt} above are given by
\begin{equation}
\label{eq:S}
\pmb{S}=S(\pmb{u}, \pmb{Q})=\pmb{W}\pmb{Q}-\pmb{Q}\pmb{W}+\xi\left(\pmb{Q}
\pmb{D}+\pmb{D}\pmb{Q}\right)+\frac{2\xi}{d}\pmb{D}-2\xi(\pmb{D}:\pmb{Q})
\left(\pmb{Q}+\frac{1}{d}\pmb{I}\right),
\end{equation}
and
\begin{equation}
\label{eq:sigmax}
\pmb{\Sigma}=\Sigma(\pmb{Q}, \pmb{H})=\pmb{Q}\pmb{H}-\pmb{H}\pmb{Q}-\xi\left(\pmb{H}\pmb{Q}+\pmb{Q}\pmb{H} \right)-\frac{2\xi}{d}\,\pmb{H}+2\xi(\pmb{Q}:\pmb{H})\,\left(\pmb{Q}+\frac{1}{d}\pmb{I}\right).
\end{equation}
with
\begin{equation}
\label{eq:D_W}
\pmb{D}=\frac{1}{2}\left(\nabla \pmb{u}+(\nabla\pmb{u})^\top\right),\quad\,\, \pmb{W}=\frac{1}{2}\left(\nabla\pmb{u}-(\nabla\pmb{u})^\top \right)
\end{equation}
representing the symmetric and skew-symmetric parts of the matrix $\nabla\pmb{u}$. Here $\pmb{S}$ denotes the rotational and stretching effects on the liquid crystal molecules generated by the flow, while the constant $\xi$ measures the degree of these effects. $\pmb{\Sigma}$ is an elastic stress tensor term~\cite{WuCavaterra2016}. The tensor $\pmb{H}$ is the molecular field corresponding to the variational derivative of the free energy ${E}_{LG}(\pmb{Q})$ and given by
\begin{equation}
\label{eq:original_system_H}
\pmb{H}=-\frac{\partial{E}_{LG}}{\partial \pmb{Q}}=L\Delta \pmb{Q}-\left[ a\pmb{Q}-b\left(\pmb{Q}^2-\frac{1}{d}\tr(\pmb{Q}^2)\pmb{I}\right) -c\tr(\pmb{Q}^2)\,\pmb{Q} \right].
\end{equation}
Notice that the last term in the definition of $\pmb{\Sigma}$,~\eqref{eq:sigmax} results in a gradient term after taking the divergence as it is the case in~\eqref{eq:original_system_ut}. Hence we can modify the pressure to include this term and instead use the modified definition of $\pmb{\Sigma}$:
\begin{equation}
\label{eq:sigma}
\pmb{\Sigma}=\Sigma(\pmb{Q}, \pmb{H})=\pmb{Q}\pmb{H}-\pmb{H}\pmb{Q}-\xi\left(\pmb{H}\pmb{Q}+\pmb{Q}\pmb{H} \right)-\frac{2\xi}{d}\,\pmb{H}+2\xi(\pmb{Q}:\pmb{H})\pmb{Q}.
\end{equation}
Indeed, as we will be concerned with Leray-Hopf solutions in the following, these definitions can be used interchangeably. In the following, we will always use definition~\eqref{eq:sigma} for $\pmb{\Sigma}$ and the accordingly modified definition of the pressure.
System~\eqref{eq:original_system_ut}--\eqref{eq:sigma} is equivalent to the Beris-Edwards model as it is shown in \cite[Section 2.1]{AbelsLiu2014}.

Our goal in this work is to provide a convergence proof for a semi-discrete numerical scheme for ~\eqref{eq:original_system_ut}--\eqref{eq:sigma}. The existence, uniqueness and regularity theory for this system have been studied in, e.g.,~\cite{AbelsLiu2014, AbelsLiu2016, WuCavaterra2016, Gonzalez2014, GUILLENGONZALEZ201584,Paicu_global_existence, paicu_energy_dissipation}. Numerical simulation and analysis of this and related models have been undertaken in, e.g.,~\cite{MeshAdaptive,Bartels_Simulation, NochettoFEM, TensorBased, Davis_FEM, MovingMesh}. Due to the system being highly nonlinear, for stability of the numerical method, it is crucial to retain a discrete version of the energy dissipation law satisfied by the system at the level of the numerical scheme. However, this often results in nonlinearly implicit schemes which require the iterative solution of a nonlinear algebraic system at every timestep. In order to circumvent this issue, the invariant energy quadratization (IEQ) method has been introduced for nonlinear gradient flows~\cite{GuillenGonzalez2013,ZHAO2017803,JiangIEQ, YANG2017691, IEQ_Cahn-Hilliard, YANG201880, YANG2017104}. The key idea is to introduce an auxiliary variable for the bulk potential term which is then discretized as an independent variable. This results in a linearly implicit scheme which is unconditionally energy-stable. A discrete version of the energy dissipation property is retained while enhancing computational efficiency.

Specifically, in the case of system~\eqref{eq:original_system_ut}-\eqref{eq:sigma}, the auxiliary variable $r$ is introduced~\cite{ZHAO2017803}:
\begin{equation}
\label{eq:auxiliary_variable_r}
r(\pmb{Q})=\sqrt{2\left(\frac{a}{2}\tr(\pmb{Q}^2) -\frac{b}{3}\tr(\pmb{Q}^3)+\frac{c}{4}\tr^2(\pmb{Q}^2)+A_0 \right)},
\end{equation}
where $A_0>0$ is a constant ensuring that $r$ is always positive for any $\pmb{Q}\in\mathbb{R}^{d\times d}$. This is possible since one can show that the bulk potential $\mathcal{F}_B(\pmb{Q})$ has a lower bound, see~\cite[Theorem 2.1]{ZHAO2017803}. 
If we then define
\begin{equation*}
\label{eq:VQ}
V(\pmb{Q})=a\pmb{Q}-b\left[\pmb{Q}^2-\frac{1}{d}\tr(\pmb{Q}^2)\pmb{I}\right]+c\tr(\pmb{Q}^2)\pmb{Q},
\end{equation*}
it follows that
\begin{equation}
\label{eq:P(Q)_definition}
\frac{\delta r(\pmb{Q})}{\delta\pmb{Q}}=\frac{V(\pmb{Q})}{r(\pmb{Q})}\coloneqq P(\pmb{Q}),
\end{equation}
for a trace-free, symmetric tensor $\pmb{Q}$. Then system \eqref{eq:original_system_ut}-\eqref{eq:original_system_Qt} can be reformulated as
\begin{subequations}
\label{eq:reformulated_pde_system}
    \begin{empheq}[left = \empheqlbrace]{align}  
&\pmb{u}_t+(\pmb{u}\cdot\nabla)\pmb{u}=-\nabla p+\mu\Delta\pmb{u}+\nabla\cdot\pmb{\Sigma}-\pmb{H}\nabla\pmb{Q},\label{eq:ut}\\
&\nabla\cdot\pmb{u}=0,\label{eq:divergence_free_reformulated}\\
&\pmb{Q}_t+\pmb{u}\cdot\pmb{Q}-\pmb{S}=M\pmb{H},\label{eq:Qt}\\
&r_t=P(\pmb{Q}):\pmb{Q}_t,\label{eq:rt}\\
&\pmb{H}=L\Delta \pmb{Q}-rP(\pmb{Q})\label{eq:H}.
\end{empheq}
\end{subequations}



In~\cite{ZHAO2017803}, the authors proposed an energy stable scheme for the reformulated system \eqref{eq:ut}-\eqref{eq:H}, and proved that it satisfies a discrete version of the energy dissipation law. However, to the best of our knowledge, there is no proof of convergence of the numerical scheme designed for the Beris-Edwards model based on the IEQ method. The main issue is that the reformulation of~\eqref{eq:original_system_ut}--\eqref{eq:sigma} to~\eqref{eq:ut}--\eqref{eq:H} is only valid at the formal level assuming solutions are smooth. However, this may not be the case for this system, given that it involves coupling to the incompressible Navier-Stokes equations. Therefore, at least in three space dimensions, at most global weak solutions can be expected. Furthermore, a priori, the auxiliary variable $r$ has less integrability than the square root of the bulk potential. While the square root of the bulk potential is expected to be in the Lebesgue space $L^3$ in space, the auxiliary variable is only expected to be in $L^2$ according to the reformulated energy dissipation law. In this work, we will show how to circumvent this issue and obtain a priori estimates for the numerical approximations which are sufficient for passing to the limit and obtaining a weak solution of~\eqref{eq:original_system_ut}--\eqref{eq:sigma}. Hence, this can also be seen as an alternative proof of existence of global weak solutions for the Beris-Edwards system. The rest of this article is structured as follows: In Section~\ref{sec:pre}, we introduce the notations and some standard results that will be used in the following.
Then we will construct and analyze a numerical scheme designed for system \eqref{eq:ut}--\eqref{eq:H} in Section~\ref{sec:numerical_scheme}. We will also provide a discrete energy dissipation law in this section. In Section~\ref{sec:convergence}, we provide the convergence argument. Finally, we will show the equivalence between weak solutions for the reformulated system and weak solutions of the original system~\eqref{eq:original_system_ut}--\eqref{eq:original_system_H}.




\section{Preliminaries}\label{sec:pre}
\subsection{Notation}
Let $\dom\subset \R^d$ be a bounded domain with $C^2$ boundary. 
We denote the norm of a Banach space $X$ as $\|\cdot\|_X$ and its dual space by $X^*$. If we omit the subscript $X$, it represents the norm of the space $L^2(\dom)$. For simplicity, when used as a subscript, we will not write the symbol $\dom$ if we refer to a function space over domain $\Omega$, i.e., $L^2=L^2(\dom)$. The inner product on $L^2$ will be denoted by $\langle\cdot, \cdot\rangle$. Vector-valued and matrix-valued functions will be denoted in bold form. 

For two vectors $\pmb{u}, \pmb{v}\in\mathbb{R}^d$, we set their inner product to be $\pmb{u}\cdot\pmb{v}=\sum_{i=1}^du_iv_i$ and for two matrices $\pmb{A}, \pmb{B}\in\mathbb{R}^{d\times d}$, we use the Frobenius inner product $\pmb{A}:\pmb{B}=\tr(\pmb{A}^\intercal \pmb{B})=\sum_{i, j =1}^d A_{ij}B_{ij}.$ The norm of matrix \pmb{A} is then given by $\lvert \pmb{A}\rvert=\lvert \pmb{A}\rvert_F=\sqrt{\pmb{A}:\pmb{A}}$. Finally, the derivatives of matrix $\pmb{A}$ are defined as a matrix, that is,  $\partial_i \pmb{A}=(\partial_i A_{jk})_{jk}$ and $\nabla\pmb{A}=(\partial_1\pmb{A}, \cdots, \partial_d\pmb{A})$. When we write $\|\pmb{A}\|$, $\|\nabla\pmb{A}\|$, we mean $\|\pmb{A}\|=\left(\int_\dom \lvert A\rvert^2\,dx\right)^{\frac{1}{2}}$ and $\|\nabla\pmb{A}\|=\left( \int_\dom\sum_{i=1}^d\lvert \partial_i \pmb{A}\rvert^2\,dx \right)^{\frac{1}{2}}$.

Throughout this paper, we will denote Sobolev spaces and Bochner spaces in standard ways, and will not tell the difference between scalar and vector value function spaces if it is clear enough from the context. In particular, we use $L^p(0,T; X)$ to denote the space of functions $f:[0,T)\to X$ which are $L^p$-integrable in the time variable $t\in [0,T)$. The dual space of $H^1_0(\Omega)$ is denoted by $H^{-1}(\Omega)$. We define $\mathcal{S}_0^d$ to be the space of traceless symmetric $\mathbb{R}^{d\times d}$ matrices,
\begin{equation*}
    \label{eq:S_0^d}
    \mathcal{S}_0^d\coloneqq \{\pmb{A}\in\mathbb{R}^{d\times d}: A_{ij}=A_{ji}, \sum_{i=1}^d A_{ii}=0, 1\leq i,j\leq d \}.
\end{equation*}
If there is no additional explanation, when we refer to a matrix-valued function $\pmb{Q}$ (including $\Qnew, \Qold, \Qsol, \Qsolsubseq$, etc.), we mean $\pmb{Q}: \dom\to \mathcal{S}_0^d$.
We will use the subscript $\sigma$  to indicate the divergence-free vector spaces, for example, 
\begin{equation*}
    \label{eq:sigma_space}
    \begin{aligned}
        C_{c,\sigma}^\infty(\Omega)=\{\pmb{\phi}\in C_c^\infty(\dom); \Div \pmb{\phi}=0\},&\quad L^2_{\sigma}(\Omega)=\{\pmb{\phi}\in L^2(\dom):\Div \pmb{\phi}=0, \pmb{\phi}\cdot\pmb{n}|_{\partial\dom}=0\}=\overline{C_c^\infty(\dom)}^{L^2(\dom)},\\
        &\quad H^1_{0,\sigma}(\dom)=H^1_0(\dom)\cap L^2_\sigma(\dom).
    \end{aligned}
\end{equation*}

We denote the Leray projector by $\mathcal{P}: L^2(\Omega)\to L^2_\sigma(\Omega)$, which is an orthogonal projection
induced by the Helmholtz-Hodge decomposition~\cite{Temam_NavierStokes} $\pmb{f}=\nabla g+\pmb{h}$ for any $\pmb{f}\in L^2(\Omega)$. Here, $g\in H^1(\Omega)$ is a scalar field, and $\pmb{h}\in L^2_\sigma(\Omega)$ is a divergence-free vector field. Then for all $\pmb{f}\in L^2(\Omega)$, it holds that $\mathcal{P}\pmb{f}=\pmb{h}$.

We will use $C$ to denote a generic constant, which might depend on parameters $\mu, a, b, c, M, L, \xi, d$, domain $\Omega$, and initial values $(\pmb{u}_{in}, \pmb{Q}_{in})$. If a constant depends on any other factors, it will be specified. The product space of two Banach spaces $X$ and $Y$ will be denoted as $X\times Y$ for all $(x,y)\in X\times Y$ where $x\in X, y\in Y$.

\subsection{Technical lemmas and definition of weak solutions}

Here we will list the technical tools that will be frequently used in the following analysis. To obtain higher order regularity of $\pmb{Q}$ in space, we recall Agmon's inequality~\cite[Lemma 4.10]{ConstantinFoias+2022}.
\begin{lemma}
    \label{lem:agmon}
    For any $f\in H^2(\Omega)\cap H^1_0(\Omega)$,
    \begin{equation}
        \label{eq:agmon}
        \|f\|_{L^\infty}\leq C\|f\|_{H^1}^{\frac{1}{2}}\|f\|_{H^2}^{\frac{1}{2}}.
    \end{equation}
\end{lemma}

The following lemma states an a priori estimate for Laplace operator \cite[Theorem 3.1.2.1]{Grisvard}.
\begin{lemma}
    \label{lem:laplace_estimate}
    There exists a constant $C$ which only depends on the diameter of $\Omega$, such that
\begin{equation}\label{eq:laplace_priori_estimate}
    \|{f}\|_{H^2}\leq C\|\Delta f\|,
\end{equation}
for all $f\in H^2(\Omega)\cap H^1_0(\Omega)$.
\end{lemma}

We will also use the Aubin-Lions lemma~\cite{MathematicsToolsForNavierStokes,Simon_Aubin-Lions}:

\begin{lemma}
    \label{lem:Aubin-Lions}
    Let $X_0\subset X_1\subset X_2$ be three Banach spaces. Assume that the embedding of $X_1$ into $X_2$ is continuous and that the embedding of $X_0$ into $X_1$ is compact. Let $p,r\in[1,\infty]$. Now if a family of functions $\mathcal{F}$ satisfies that for any $f\in\mathcal{F}$, 
    \begin{equation*}
        \label{eq:Compactness_condition}
        f\in L^p([0,T); X_0),\quad\quad \frac{df}{dt}\in L^r([0,T); X_2)
    \end{equation*}
    Then if $p<\infty$, $\mathcal{F}$ is a compact family in $L^p([0,T); X_1)$. If $p=\infty$, then $\mathcal{F}$ is a compact family in $C([0,T); X_1)$.
\end{lemma}

\begin{definition}
	\label{def:weakreformulation_original_system}
		By a weak solution of system ~\eqref{eq:ut} to \eqref{eq:H}, we mean a triple $(\pmb{u},\pmb{Q}, \pmb{H})$, with $\pmb{u}:[0,T)\times\Omega\to \mathbb{R}^{d}$, $\pmb{Q}:[0,T)\times\Omega\to\R^{d\times d}$ and $\pmb{H}:[0,T)\times\Omega\to\R^{d\times d}$ which satisfy
		\begin{enumerate}[(i)]
		    \item $\pmb{Q}(t,x)$ is trace-free and symmetric and $\pmb{u}(t,x)$ is divergence free for almost every $(t,x)$.
		    \item They attain the initial values
		    \begin{equation*}
    \pmb{Q}(0,x)=\pmb{Q}_0(x)\in H^1(\Omega),\quad \pmb{u}(0,x)=\pmb{u}_0(x)\in L^2(\Omega), \quad \langle\pmb{u}_0, \nabla{\psi}\rangle=0,
    \end{equation*}
     for any smooth function ${\psi}\in C^\infty_c(\dom)$.
    
    \item The triple $(\pmb{u}, \pmb{Q}, \pmb{H})$  satisfies the regularity condition
	\begin{equation*}
	\pmb{Q}\in L^\infty(0,T;H^1(\dom))\cap L^2(0,T;H^2(\dom)), \quad\pmb{u}\in L^\infty(0,T;L^2(\dom))\cap L^2(0,T;H^1(\dom)), \quad\pmb{H}\in L^2([0,T)\times\Omega).
	\end{equation*}
	
	\item $(\pmb{u},\pmb{Q}, \pmb{H})$ satisfy the weak formulations
	\begin{subequations}
	    
\label{eq:weak_formulation_solution}
		\begin{equation}
	    \label{eq:weakformu}
	    \begin{aligned}
	         &\quad\,\,\int_0^T\int_\dom \pmb{u}\,\partial_t\pmb{\psi}\,dxdt-\int_\Omega \pmb{u}(T,x)\cdot\pmb{\psi}(T,x) \,dx +\int_{\Omega}\pmb{u}_0(x)\cdot\pmb{\psi}(0,x)\, dx+\int_0^T\int_\dom\sum_{i,j=1}^d u_iu_j\partial_i\psi_j\,dxdt\\
	         &=\int_0^T\int_\dom\left[ \left(\pmb{Q}\pmb{H}-\pmb{H}\pmb{Q}\right)-\xi\left(\pmb{H}\pmb{Q}+\pmb{Q}\pmb{H} \right)-\frac{2\xi}{d}\,\pmb{H}+2\xi(\pmb{Q}:\pmb{H})\,\left(\pmb{Q}+\frac{1}{3}\pmb{I}\right)\right]:\nabla\pmb{\psi}\,dxdt\\
	         &\quad\,\,+\mu\int_0^T\int_\dom \nabla\pmb{u}:\nabla\pmb{\psi}\,dxdt+\int_0^T\int_\dom (\pmb{H}\nabla\pmb{Q})\cdot\pmb{\psi}\,dxdt,
	    \end{aligned}
	\end{equation}
	
	\begin{equation}
	\label{eq:weakformQr}
	\begin{split}
	&\int_0^T\int_\Omega \pmb{Q} :\partial_t\pmb{\varphi} \,dx dt-\int_\Omega \pmb{Q}(T,x):\pmb{\varphi}(T,x) \,dx +\int_{\Omega}\pmb{Q}_0(x):\pmb{\varphi}(0,x)\, dx + \int_0^T\int_\dom \pmb{Q}:(\pmb{u}\cdot\nabla \pmb{\varphi})\,dxdt\\
	&+\int_0^T\int_\dom \left[\pmb{W}\pmb{Q}-\pmb{Q}\pmb{W}+\xi\left(\pmb{Q}
  \pmb{D}+\pmb{D}\pmb{Q}\right)+\frac{2\xi}{d}\pmb{D}-2\xi(\pmb{D}:\pmb{Q})\pmb{Q}\right]:\pmb{\varphi}\,dxdt \\
	& = -\int_0^T\int_\Omega M\pmb{H}:\pmb{\varphi}\,dxdt,
	\end{split}
	\end{equation}
	\begin{equation}
	    \label{eq:weakformH}
	    \begin{aligned}
	    \int_0^T\int_\Omega \pmb{H}:\pmb{\varphi}\,dxdt=&-\int_0^T\int_\Omega \left(L \sum_{i,j=1}^d\Grad Q_{ij}\cdot\Grad\varphi_{ij}  \right)\,dxdt\\
	    &-\int_0^T\int_\Omega \left( a\pmb{Q}-b\left( (\pmb{Q}^2)-\frac{1}{3}\tr(\pmb{Q}^2)+c\tr(\pmb{Q}^2)\pmb{Q} \right) \right):\pmb{\varphi} \, dxdt,
	    \end{aligned}
	\end{equation}
		\end{subequations}
	 for all smooth divergence-free function $\pmb{\psi}:[0,T)\times\Omega\to\mathbb{R}^{d}$ and all trace-free symmetric matrix function $\pmb{\varphi} = (\varphi_{ij})_{i,j=1}^d:[0,T)\times \Omega\to \R^{d\times d}$ which are compactly supported within $\dom$ for every $t\in [0,T]$.
    
		\end{enumerate}


\end{definition}

\begin{definition}
	\label{def:weakreformulation}
		By a weak solution of system ~\eqref{eq:ut} to \eqref{eq:H}, we mean a quadruple $(\pmb{u},\pmb{Q},\pmb{H},r)$, with $\pmb{u}:[0,T)\times\Omega\to \mathbb{R}^{d}$, $\pmb{Q}:[0,T)\times\Omega\to\R^{d\times d}$, $\pmb{H}:[0,T)\times\Omega\to\R^{d\times d}$ and $r:[0,T)\times\dom\to\R$, 
		
			\begin{enumerate}[(i)]
			\item $\pmb{Q}(t,x)$ is trace-free and symmetric and $\pmb{u}(t,x)$ is divergence free for almost every $(t,x)$,
			\item They attain the initial values
		 \begin{equation*}
		\pmb{Q}(0,x)=\pmb{Q}_0(x)\in H^1(\Omega),\quad \pmb{u}(0,x)=\pmb{u}_0(x)\in L^2(\Omega),  \quad r(0,x)=r\left(Q_0(x)\right),\quad \langle\pmb{u}_0, \nabla{\psi}\rangle=0,
		\end{equation*}
		for any smooth function ${\psi}\in C^\infty_c(\dom)$.

			\item $(\pmb{u}, \pmb{Q}, r)$  satisfy the regularity condition
			\begin{align*}	\label{eq:regularity}
		\pmb{Q}\in L^\infty(0,T;H^1(\dom&))\cap L^2(0,T;H^2(\dom)), \quad\pmb{u}\in L^\infty(0,T;L^2(\dom))\cap L^2(0,T;H^1(\dom)), 
		\\
		&\pmb{H}\in L^2(0,T;L^2(\Omega)),\quad
		r\in L^\infty(0,T;L^2(\dom))
		\end{align*}
			
			\item $(\pmb{u},\pmb{Q},\pmb{H},r)$ satisfy the weak formulations
			\begin{subequations}
			  \label{eq:reformlated_weakform}  
		
			\begin{equation}
			\label{eq:re_weakformu}
			\begin{aligned}
			&\quad\,\,\int_0^T\int_\dom \pmb{u}\,\partial_t\pmb{\psi}\,dxdt-\int_\Omega \pmb{u}(T,x)\cdot\pmb{\psi}(T,x) \,dx +\int_{\Omega}\pmb{u}_0(x)\cdot\pmb{\psi}(0,x)\, dx+\int_0^T\int_\dom\sum_{i,j=1}^d u_iu_j\partial_i\psi_j\,dxdt\\
			&=\int_0^T\int_\dom\left[ \left(\pmb{Q}\pmb{H}-\pmb{H}\pmb{Q}\right)-\xi\left(\pmb{H}\pmb{Q}+\pmb{Q}\pmb{H} \right)-\frac{2\xi}{d}\,\pmb{H}+2\xi(\pmb{Q}:\pmb{H})\,\left(\pmb{Q}+\frac{1}{3}\pmb{I}\right)\right]:\nabla\pmb{\psi}\,dxdt\\
			&\quad\,\,+\mu\int_0^T\int_\dom \nabla\pmb{u}:\nabla\pmb{\psi}\,dxdt+\int_0^T\int_\dom (\pmb{H}\nabla\pmb{Q})\cdot\pmb{\psi}\,dxdt,
			\end{aligned}
			\end{equation}
			
			\begin{equation}
			\label{eq:re_weakformQr}
			\begin{split}
			&\int_0^T\int_\Omega \pmb{Q} :\partial_t\pmb{\varphi} \,dx dt-\int_\Omega \pmb{Q}(T,x):\pmb{\varphi}(T,x) \,dx +\int_{\Omega}\pmb{Q}_0(x):\pmb{\varphi}(0,x)\, dx + \int_0^T\int_\dom \pmb{Q}:(\pmb{u}\cdot\nabla \pmb{\varphi})\,dxdt\\
			&+\int_0^T\int_\dom \left[\pmb{W}\pmb{Q}-\pmb{Q}\pmb{W}+\xi\left(\pmb{Q}
			\pmb{D}+\pmb{D}\pmb{Q}\right)+\frac{2\xi}{d}\pmb{D}-2\xi(\pmb{D}:\pmb{Q})\pmb{Q}\right]:\pmb{\varphi}\,dxdt \\
			&=-\int_0^T\int_\Omega M\pmb{H}:\pmb{\varphi}\,dxdt
			\end{split}
			\end{equation}
			\begin{equation}
			\label{eq:re_weakr}
			\int_0^T\int_\dom r\, \phi_t dxdt - \int_\dom r(T,x) \phi(T,x) dx +\int_\dom r_0(x)\phi(0,x) dx =- \int_0^T\int_\dom P(Q):Q_t \,\phi \,dx dt,
			\end{equation}
			and
			\begin{equation}
	    \label{eq:re_weakformHr}
	    \int_0^T\int_\Omega \pmb{H}:\pmb{\varphi}\,dxdt=-\int_0^T\int_\Omega \left(L \sum_{i,j=1}^d\Grad Q_{ij}\cdot\Grad\varphi_{ij}  \right)\,dxdt-\int_0^T\int_\Omega rP(\pmb{Q}):\pmb{\varphi}\,dxdt,
	\end{equation}
		\end{subequations}
			for all smooth divergence-free function $\pmb{\psi}:[0,T)\times\Omega\to\mathbb{R}^{d}$, all trace-free symmetric matrix function $\pmb{\varphi} = (\varphi_{ij})_{i,j=1}^d:[0,T)\times \Omega\to \R^{d\times d}$ and smooth function $\phi:[0,T)\times \dom\to \R$ which are compactly supported within $\dom$ for every $t\in [0,T]$.

		\end{enumerate}


\end{definition}



Then for the treatment of the convection term, we consider a bilinear form
\begin{equation}
    \label{eq:convection_B_term}
    B(\pmb{u}, \pmb{v}) = (\pmb{u}\cdot\nabla)\pmb{v}+\frac{1}{2}(\nabla\cdot\pmb{u})\pmb{v}.
\end{equation}
It is not hard to verify the following properties of $B$ (See \cite{ NochettoGauge_FEM, ShenJie_error_first_order,Temam_NavierStokes} and the references therein).
\begin{lemma}
    \label{lem:properties_B}
    We define the trilinear form 
    \begin{equation}
        \label{eq:Tilde_B}
        \Tilde{B}(\pmb{u}, \pmb{v}, \pmb{\omega})=\langle B(\pmb{u}, \pmb{v}), \pmb{\omega}\rangle.
    \end{equation}
    Then 
    \begin{equation}
        \label{eq:anti_sym_B}
        \Tilde{B}(\pmb{u}, \pmb{v}, \pmb{\omega})=-\Tilde{B}(\pmb{u}, \pmb{\omega}, \pmb{v}),
    \end{equation}
    for all $\pmb{u}\in L^2(\Omega)$ with $L^2(\dom)$-integrable divergence, and $\pmb{v}, \pmb{\omega}\in H^1_0(\Omega)$. Moreover, $\Tilde{B}(\pmb{u}, \pmb{v}, \pmb{v})=0$.
\end{lemma}

The following cancellation property will play a key role in deducing the discrete energy dissipation law in the next section.

\begin{lemma}
    \label{lem:Sigma_S_cancellation}
    For any $\pmb{u}\in H^1_{0}(\Omega)$, we have
    \begin{equation}
        \label{eq:Sigma_S_cancellation_property}
        \left\langle \nabla \pmb{u}, \Sigma(\pmb{Q}, \pmb{H})\right\rangle+\left\langle\pmb{H}, S(\pmb{Q}, \pmb{u})\right\rangle=0,
    \end{equation}
    for all symmetric trace-free matrices $\pmb{Q}\in H^2(\Omega), \pmb{H}\in L^2(\Omega)$.
\end{lemma}
\begin{proof}
From definitions~\eqref{eq:S} and~\eqref{eq:sigma}, we have
\begin{equation*}
    \label{eq:H_S}
    \left\langle  \pmb{H}, S(\pmb{Q}, \pmb{u})\right\rangle=\left\langle \pmb{H}, \pmb{W}\pmb{Q}-\pmb{Q}\pmb{W}+\xi\left(\pmb{Q}
    \pmb{D}+\pmb{D}\pmb{Q}\right)+\frac{2\xi}{d}\pmb{D}-2\xi(\pmb{D}:\pmb{Q})
    \left(\pmb{Q}+\frac{1}{d}\pmb{I}\right) \right\rangle,
\end{equation*}
\begin{equation*}
    \label{eq:Sigma_Gradientu}
    \begin{aligned}
        \left\langle \nabla\pmb{u}, \Sigma(\pmb{Q}, \pmb{H})\right\rangle&=\left\langle \nabla\pmb{u}, \pmb{Q}\pmb{H}-\pmb{H}\pmb{Q}-\xi\left(\pmb{H}\pmb{Q}+\pmb{Q}\pmb{H} \right)-\frac{2\xi}{d}\pmb{H}+2\xi(\pmb{Q}:\pmb{H})\pmb{Q}\right\rangle.
    \end{aligned}
\end{equation*}
Comparing these terms and utilizing the symmetry and trace-free property of $\pmb{H}$ and $\pmb{Q}$, we observe that
\begin{equation*}
    \label{eq:uSSH_term_1}
    \begin{aligned}
          \left\langle \pmb{H}, \pmb{W}\pmb{Q}-\pmb{Q}\pmb{W} \right\rangle&=\int_\dom \sum_{i, j, k =1}^d H_{ij}(W_{ik}Q_{kj}-Q_{ik}W_{kj})\\
          &=\int_\dom \sum_{i,j,k=1}^d (W_{ik}H_{ij}Q_{jk}-W_{kj}Q_{ki}H_{ij})=\left\langle  \pmb{W},  \pmb{H}\pmb{Q}-\pmb{Q}\pmb{H}\right\rangle=-\left\langle \nabla\pmb{u}, \pmb{Q}\pmb{H}-\pmb{H}\pmb{Q}\right\rangle,
    \end{aligned}
\end{equation*}
\begin{equation*}
    \label{eq:uSSH_term_2}
    \begin{aligned}
          \left\langle \pmb{H}, \xi(\pmb{Q}\pmb{D}+\pmb{D}\pmb{Q}) \right\rangle&=\int_\dom \sum_{i, j, k =1}^d \xi H_{ij}(Q_{ik}D_{kj}+D_{ik}Q_{kj})\\
          &=\int_\dom \sum_{i,j,k=1}^d \xi(D_{ik}H_{ij}Q_{jk}+D_{kj}Q_{ki}H_{ij})\\
          &=\left\langle  \pmb{D},  \xi(\pmb{H}\pmb{Q}+\pmb{Q}\pmb{H})\right\rangle=\xi\left\langle \nabla\pmb{u}, \pmb{H}\pmb{Q}+\pmb{Q}\pmb{H}\right\rangle,
    \end{aligned}
\end{equation*}
\begin{equation*}
    \label{eq:uSSH_term_3}
    \left\langle \pmb{H}, \frac{2\xi}{d}\pmb{D}\right\rangle = \left\langle \nabla\pmb{u}, \frac{2\xi}{d}\pmb{H}\right\rangle,\quad \left\langle \pmb{H}, -2\xi(\pmb{D}:\pmb{Q})\pmb{Q}\right\rangle=-2\xi\int_{\dom}(\nabla\pmb{u}:\pmb{Q})(\pmb{H}:\pmb{Q})dx=-\left\langle \nabla\pmb{u}, 2\xi(\pmb{Q}:\pmb{H})\pmb{Q}\right\rangle,
\end{equation*}
\begin{equation*}
    \label{eq:uSSH_term_4}
    \left\langle \pmb{H}, \frac{2\xi}{d}(\pmb{D}:\pmb{Q})\pmb{I}\right\rangle=\frac{2\xi}{d}\int_{\dom}(\pmb{D}:\pmb{Q})\tr(\pmb{H})dx =0.
\end{equation*}
From these calculations, we can conclude that \eqref{eq:Sigma_S_cancellation_property} holds true.
\end{proof}

We also recall the following lemma from~\cite[Theorem 4.1]{GWY2020}, establishing Lipschitz continuity of $P$. We will use this lemma to pass to the limit in the numerical approximations introduced below and obtain convergence to a weak solution as in Definition~\ref{def:weakreformulation}. 
\begin{lemma}\label{lem:P_lipschitz}
    The function $P$ is Lipschitz continuous, that is, there exists constant $\Tilde{L}>0$ such that for any matrix $\pmb{Q}, \delta\pmb{Q}\in\mathbb{R}^{3\times 3}$, 
    \begin{equation}
        \label{eq:P_lipschitz}
        \lvert P(\pmb{Q}+\delta\pmb{Q})-P(\pmb{Q})\rvert\leq \Tilde{L}\,\lvert \delta\pmb{Q}\rvert.
    \end{equation}
\end{lemma}

\section{Construction and Analysis of the Numerical Scheme}\label{sec:numerical_scheme}

We start by describing the first-order semi-discrete numerical scheme for system \eqref{eq:ut}--\eqref{eq:H}. It is based on the projection method, a fractional step method widely used for the numerical approximation of the Navier-Stokes equations~\cite{Chorin1967, ShenJie_error_first_order, Temam_NavierStokes}. It consists of two steps. Let $\Delta t>0$ be the time step size. 

Given initial data $(\pmb{u}^0, \pmb{Q}^0, p^0)\in H^1_0(\Omega)\times \left(H^1_0(\Omega)\cap H^2(\Omega)\right)\times H^2(\Omega)$, 
we set $\pmb{P}^0=P(\pmb{Q}^0), r^0=r(\pmb{Q}^0)$ and $(\pmb{u}^{-1}, \pmb{Q}^{-1}, {p}^{-1}, r^{-1})=(\pmb{u}^{0}, \pmb{Q}^{0}, {p}^{0}, r^{0})$. Then for $n=0,1,\dots$, we update $(\pmb{u}^{n+1}, \pmb{Q}^{n+1}, \pmb{p}^{n+1}, r^{n+1})$ through the following two steps.

\begin{enumerate}
	\item[\textbf{Step 1}] 
	
	Given $(\pmb{u}^{n}, \pmb{Q}^{n}, p^n, r^{n})\in H_0^1(\Omega)\times \left(H^1_0(\Omega)\cap H^2(\Omega)\right)\times H^2(\Omega)\times L^2(\Omega)$, we seek $(\uint, \pmb{Q}^{n+1}, r^{n+1})\in H_0^1(\Omega)\times \left(H^1_0(\Omega)\cap H^2(\Omega)\right)\times L^2(\Omega)$ as a weak solution of the following system with boundary conditions $\uint|_{\partial\dom}=0,\, \Qnew|_{\partial\dom}=0$,
	\begin{subequations}\label{eq:uQevolution}
		\begin{align}
		\left\langle\Dtuone, \pmb{\psi}\!\right\rangle + \Tilde{B}(\uold, \uint, \pmb{\psi}) &=-\left\langle\nabla p^n, \pmb{\psi}\right\rangle -\mu\left\langle  \nabla\uint, \nabla\pmb{\psi}\right\rangle-\left\langle \sigmanew, \nabla\pmb{\psi}\right\rangle \label{eq:Dtu_step1}, \\
   &\quad -\left\langle\Hnew\nabla\Qold, \pmb{\psi}\right\rangle\notag\\
		  \left\langle\frac{\Qnew\!\!-\Qold}{\Delta t}, \testvar\!\right\rangle  + \left\langle\uint\!\cdot \nabla \Qold, \testvar\right\rangle  &=\left\langle\Snew, \testvar\right\rangle +M\left\langle\Hnew, \testvar\right\rangle,\label{eq:DtQ_step1}\\ 
		 r^{n+1}-r^n&=\Pold:(\Qnew-\Qold)\label{eq:Dtr_step1},\\
	   \left\langle \Hnew, \pmb{\phi}\right\rangle &= -L\left\langle \nabla \Qnew,\nabla\pmb{\phi}\right\rangle-\left\langle r^{n+1} \Pold, \pmb{\phi} \right\rangle\label{eq:Hn+1}	
    \end{align}
	\end{subequations}
	for all smooth vector-valued function $\pmb{\psi}$ and smooth matrix-valued function $\pmb{\varphi}, \pmb{\phi}$ with compact support in $[0,T)\times\Omega$.
	where
	\begin{equation}
	\label{eq:SnSigman}
	\Snew={s}(\uint,\Qold),\quad\quad \sigmanew={\Sigma}(\Qold,\Hnew),\quad\quad \Pold=P(\Qold)\quad\quad \text{ for all } n\geq0.
	\end{equation}

	\item[\textbf{Step 2}] 
	
	Then we define $(\unew, p^{n+1})\in H^1(\Omega)\times H^2(\Omega)$ through the following equations with boundary condition $\unew\cdot\pmb{n}|_{\partial\Omega}=0$,
	\begin{subequations}\label{eq:pressureupdate}
		\begin{align}
		&\Dtutwo= -2\,(\nabla p^{n+1}-\nabla p^n)\label{eq:scheme_projection}, \\  &\left\langle\unew, \nabla\pmb{\psi} \right\rangle =0\label{eq:scheme_divergencefree},
		\end{align}
	\end{subequations}
	for all smooth vector-valued function $\pmb{\psi}$ with compact support in $[0,T)\times \Omega$.

\end{enumerate}

\begin{remark}\label{rem:Step2_Helmholtz_decomposition}
	The second step can be understood as applying the Helmholtz decomposition to $\uint$, in particular, $\unew=\mathcal{P}{\uint}$. 
\end{remark}
Here $s= s(\pmb{u},\pmb{Q})$ is given by
\begin{equation}
\label{eq:numerical_s}
s(\pmb{u}, \pmb{Q})\coloneqq{S}(\pmb{u}, \pmb{Q})-\frac{2\xi}{d^2}(\Div \pmb{u})\,\pmb{I}.
\end{equation}
Clearly, if $\pmb{u}$ is divergence free, this definition coincides with the definition of $S$ in~\eqref{eq:S}. However, the velocity field $\uint$ obtained in the first step of the scheme is not necessarily divergence free and hence $S$ may not be trace-free, a fact which is needed to show that the scheme conserves the trace-free properties of $\pmb{Q}$ and $\pmb{H}$, as we will see later. From the proof of Lemma~\ref{lem:Sigma_S_cancellation}, we notice that the trace-free property of $\pmb{H}$ is in fact necessary for obtaining the cancellation property~\eqref{eq:Sigma_S_cancellation_property}, which in turn is needed for showing the discrete energy balance.


Then the following version of Lemma~\ref{lem:Sigma_S_cancellation} holds:
\begin{lemma}
    \label{lem:S_Sigma_cancellation_numerical}
    For any $\pmb{u}\in H^1_0(\Omega)$, we have
    \begin{equation}
    \label{eq:S_Sigma_cancellation_numerical}
        \left\langle \nabla \pmb{u}, \Sigma(\pmb{Q}, \pmb{H})\right\rangle+\left\langle\pmb{H}, s( \pmb{u}, \pmb{Q})\right\rangle=0,
    \end{equation}
    for every symmetric trace-free matrix $\pmb{Q}\in H^1_0(\Omega)\cap H^2(\Omega)$, $  \pmb{H}\in L^2(\Omega)$.
\end{lemma}
\begin{proof}
	The proof is the same as the proof of Lemma~\ref{lem:Sigma_S_cancellation} after noting that $\frac{2\xi}{d^2}(\Div u)\tr \pmb{H}=0$.
\end{proof}

\subsection{Well-posedness of the scheme}
First, we need to guarantee a solution of~\eqref{eq:uQevolution}--\eqref{eq:pressureupdate} with the required properties exists at every step $n$. We start by noting that the scheme preserves the trace-free and symmetry property of $\pmb{Q}$ and $\pmb{H}$, i.e., if $\Qold$ is trace-free and symmetric, then $\Qnew$ and $\pmb{H}^{n+1}$ will be also. Since the second step of the scheme does not modify $\pmb{Q}$ and $\pmb{H}$, we only need to consider the first step:
%

\begin{lemma}
    \label{lem:Q-tensor_symmetry_trace-free}
    If $\Qold$ is trace-free and symmetric, then $\Qnew$ and $\pmb{H}^{n+1}$ computed through \eqref{eq:uQevolution} are also trace-free and symmetric.
\end{lemma}
\begin{proof}
	
	We use $\tr(\pmb{Q}^{n+1})\mathbf{I}$ (where $\mathbf{I}$ is the $d\times d$ identity matrix) as a test function in~\eqref{eq:DtQ_step1}:
	\begin{equation*}
	\left\langle\frac{\Qnew-\Qold}{\Delta t},\tr(\pmb{Q}^{n+1})\mathbf{I}\right\rangle  + \left\langle\uint\cdot \nabla \Qold,\tr(\pmb{Q}^{n+1})\mathbf{I}\right\rangle - \left\langle\Snew,\tr(\pmb{Q}^{n+1})\mathbf{I}\right\rangle = M\left\langle\Hnew, \tr(\pmb{Q}^{n+1})\mathbf{I}\right\rangle
	\end{equation*}
	which can be rewritten as
	\begin{equation*}
	\left\langle\frac{\tr(\Qnew)-\tr(\Qold)}{\Delta t},\tr(\pmb{Q}^{n+1})\right\rangle  + \left\langle\uint\cdot \nabla \tr(\Qold),\tr(\pmb{Q}^{n+1})\right\rangle - \left\langle\tr(\Snew),\tr(\pmb{Q}^{n+1})\right\rangle = M\left\langle\tr(\Hnew), \tr(\pmb{Q}^{n+1})\right\rangle.
	\end{equation*}
	By assumption, $\Qold$ is trace-free, hence this becomes
	\begin{equation}\label{eq:trQ1}
	\frac{1}{\Delta t}\norm{\tr(\Qnew)}^2 - \left\langle\tr(\Snew),\tr(\pmb{Q}^{n+1})\right\rangle = M\left\langle\tr(\Hnew), \tr(\pmb{Q}^{n+1})\right\rangle.
	\end{equation}	
    From the definition of $\Snew$ in \eqref{eq:numerical_s} and \eqref{eq:SnSigman}, it follows that
    \begin{equation}
        \label{eq:trace_s}
        \begin{aligned}
            \left\langle\tr(\Snew),\phi\right\rangle&=\left\langle\tr\left( S(\uint, \Qold) \right), \phi\right\rangle-\left\langle\frac{2\xi}{d^2}(\Div \uint)\tr(\pmb{I}), \phi\right\rangle\\
            &=\big\langle\tr(\Wint\Qold-\Qold\Wint)+\xi\tr(\Qold\Dint+\Dint\Qold)+\frac{2\xi}{d}\tr(\Dint)\\
            &\quad\,\,-2\xi(\Dint:\Qold)\tr(\Qold)-\frac{2\xi}{d}(\Dint:\Qold)\tr(\pmb{I})-\frac{2\xi}{d}(\Div\uint), \phi\big\rangle=0
        \end{aligned}
    \end{equation}
    for any test function $\phi: \dom\to \R$ with zero trace and contained in $H^2(\dom)$. In order to deal with the last term, we take $\tr(\Qnew)\mathbf{I}$ as a test function in~\eqref{eq:Hn+1}:
    \begin{equation*}
    \left\langle \Hnew, \tr(\Qnew)\mathbf{I}\right\rangle = -L\left\langle \nabla \Qnew,\nabla(\tr(\Qnew)\mathbf{I})\right\rangle-\left\langle r^{n+1} \Pold, \tr(\Qnew)\mathbf{I} \right\rangle,
    \end{equation*}
    which again, we can write as
    \begin{equation*}
    \left\langle \tr(\Hnew), \tr(\Qnew)\right\rangle = -L\norm{\Grad\tr(\Qnew)}^2-\left\langle r^{n+1} \tr(\Pold), \tr(\Qnew)\right\rangle.
    \end{equation*}
    Since $\tr(\Pold)=0$, we get
    \begin{equation*}
    \left\langle \tr(\Hnew), \tr(\Qnew)\right\rangle = -L\norm{\Grad\tr(\Qnew)}^2.
    \end{equation*}
    Plugging this into~\eqref{eq:trQ1}, we obtain
    	\begin{equation*}
    \frac{1}{\Delta t}\norm{\tr(\Qnew)}^2 - \left\langle\tr(\Snew),\tr(\pmb{Q}^{n+1})\right\rangle = -ML\norm{\Grad\tr(\Qnew)}^2\leq 0
    \end{equation*}
    and so $\tr(\Qnew)=0$ almost everywhere.

    For the symmetry, we consider $\pmb{Z}^{n+1}=\Qnew-(\Qnew)^{\top}$ as a test function in \eqref{eq:DtQ_step1}:
    \begin{align*}
    &\left\langle\frac{\Qnew-\Qold}{\Delta t},\Qnew-(\Qnew)^{\top}\right\rangle  + \left\langle\uint\cdot \nabla \Qold,\Qnew-(\Qnew)^{\top}\right\rangle - \left\langle\Snew,\Qnew-(\Qnew)^{\top}\right\rangle\\
    & \quad = M\left\langle\Hnew, \Qnew-(\Qnew)^{\top}\right\rangle
    \end{align*}
    which can be rewritten as
    \begin{align*}
  &\frac12\left\langle\frac{\Qnew-(\Qnew)^\top-(\Qold-(\Qold)^\top)}{\Delta t},\Qnew-(\Qnew)^{\top}\right\rangle  + \frac12\left\langle\uint\cdot \nabla (\Qold-(\Qold)^\top),\Qnew-(\Qnew)^{\top}\right\rangle\\
  &\quad  - \frac12\left\langle\Snew-(\Snew)^\top,\Qnew-(\Qnew)^{\top}\right\rangle = \frac{M}{2}\left\langle\Hnew-(\Hnew)^\top, \Qnew-(\Qnew)^{\top}\right\rangle
    \end{align*}
    Since $\Qold$ is assumed to be symmetric, a simple calculation reveals that $\Snew$ is also symmetric and so the previous identity simplifies to
    \begin{equation}\label{eq:symmQ}
    \frac{1}{2\Delta t}\norm{\Qnew-(\Qnew)^\top}^2  = \frac{M}{2}\left\langle\Hnew-(\Hnew)^\top, \Qnew-(\Qnew)^{\top}\right\rangle.
    \end{equation}
    We use $\pmb{Z}^{n+1}$ as a test function in the equation for $\Hnew$, equation~\eqref{eq:Hn+1}:
    \begin{equation*}
    \left\langle \Hnew,\Qnew-(\Qnew)^\top\right\rangle = -L\left\langle \nabla \Qnew,\nabla(\Qnew-(\Qnew)^\top)\right\rangle-\left\langle r^{n+1} \Pold, \Qnew-(\Qnew)^\top \right\rangle,
    \end{equation*}
    which we can rewrite as
        \begin{multline*}
    \frac12\left\langle \Hnew-(\Hnew)^\top,\Qnew-(\Qnew)^\top\right\rangle\\
     = -\frac{L}{2}\norm{ \nabla (\Qnew-(\Qnew)^\top)}^2-\frac12\left\langle r^{n+1} (\Pold-(\Pold)^\top), \Qnew-(\Qnew)^\top \right\rangle,
    \end{multline*}
    which noticing that $\Pold$ is symmetric since $\Qold$ is, becomes
    \begin{equation*}
    \frac12\left\langle \Hnew-(\Hnew)^\top,\Qnew-(\Qnew)^\top\right\rangle
    = -\frac{L}{2}\norm{ \nabla (\Qnew-(\Qnew)^\top)}^2.
    \end{equation*}
    Thus~\eqref{eq:symmQ} becomes
    \begin{equation*}
    \frac{1}{2\Delta t}\norm{\Qnew-(\Qnew)^\top}^2  = -\frac{LM}{2}\norm{ \nabla (\Qnew-(\Qnew)^\top)}^2\leq 0.
    \end{equation*}
    This implies that $\Qnew$ is symmetric.
\end{proof}
Next, we turn to the solvability of our numerical scheme, that is, given $(\uold, \Qold, p^n, \pmb{H}^n)\in H^1(\Omega)\times\left( H^2(\Omega)\cap H_0^1(\Omega)\right)\times H^2(\Omega)\times L^2(\Omega)$, whether there exists $(\unew, \Qnew, p^{n+1}, \pmb{H}^{n+1})\in H^1(\Omega)\times\left( H^2(\Omega)\cap H_0^1(\Omega)\right)\times H^2(\Omega)\times L^2(\Omega)$ solving equations \eqref{eq:uQevolution}-\eqref{eq:pressureupdate}. To see this, we will rewrite the scheme into a more straightforward form to implement and analyze.
From \eqref{eq:Dtr_step1}, we can express $r^{n+1}$ in terms of $\pmb{Q}^n$, $\Qnew$ and $r^n$ as
\begin{equation*}
    \label{eq:rn+1_rewrite}
    r^{n+1}=r^n+\Pold:(\pmb{Q}^{n+1}-\pmb{Q}^n).
\end{equation*}
Substituting $r^{n+1}$ into the formula for $\pmb{H}^{n+1}$ in \eqref{eq:Hn+1}, we obtain
\begin{equation}
    \label{eq:Hn+1_rewrite}
    \left\langle \Hnew, \pmb{\phi}\right\rangle=-L\left\langle \nabla \Qnew,\nabla\pmb{\phi}\right\rangle-\left\langle (\Pold:\Qnew) \Pold, \pmb{\phi} \right\rangle+\left\langle \pmb{F}^n, \pmb{\phi}\right\rangle
\end{equation}
where we denoted $(\Pold:\Qold)\,\Pold-r^n\Pold$ by $\pmb{F}^n$. Then we consider the following  problem: Given $(\uold,\Qold,p^n,\Hold,r^n)$, we want to find a unique $\begin{pmatrix}
    \pmb{u},  \pmb{Q}, \pmb{H}
\end{pmatrix}\in H_0^1(\Omega)\times H_0^1(\Omega)\times L^2(\Omega)$ such that
\begin{equation}
    \label{eq:Weak_formulation_well-posedness}
   a_{n+1}\left(\begin{pmatrix}
    \pmb{u}, \pmb{Q}, \pmb{H}
\end{pmatrix},\begin{pmatrix}
    \pmb{\psi}, \testvar, \pmb{\phi}\end{pmatrix}\right)=f_n \left(\begin{pmatrix}
    \pmb{\psi}, \testvar , \pmb{\phi}\end{pmatrix}\right)
\end{equation}
holds for all $\begin{pmatrix}
    \pmb{\psi}, \testvar,  \pmb{\phi}\end{pmatrix}\in H_0^1(\Omega)\times H_0^1(\Omega)\times L^2(\Omega)$. Here the bilinear form $a_{n+1}(\cdot, \cdot): \left(H^1_0({\dom})\times H^1_0({\dom})\times L^2(\dom)\right)\times\left(H^1_0({\dom})\times H^1_0({\dom})\times L^2(\dom)\right)\to\mathbb{R}$ is defined as:
    \begin{subequations}
    \label{eq:bilinear_form}
    
\begin{equation}
    \label{eq:bilinear_form_u_Q_H}
    \begin{aligned}
   &\quad\,\, a_{n+1}\left(\begin{pmatrix}
    \pmb{u}, \pmb{Q},  \pmb{H}
\end{pmatrix},\begin{pmatrix}
    \pmb{\psi}, \testvar, \pmb{\phi}\end{pmatrix}\right)\\
    &=
  \frac{1}{\Delta t}\int_\dom \pmb{u}\cdot\pmb{\psi}\,dx+\Tilde{B}(\uold, \pmb{u}, \pmb{\psi})+\mu\int_\dom \nabla\pmb{u}\cdot\nabla\pmb{\psi}\,dx+\int_\dom\Sigma(\Qold, \pmb{H}):\nabla\pmb{\psi}\,dx+\int_\Omega \pmb{H}\nabla \Qold\cdot \pmb{\psi}\,dx\\
  &\quad\,\,-\frac{1}{\Delta t}\int_\Omega \pmb{Q}:\pmb{\phi}\,dx-\int_\Omega\left(\pmb{u}\cdot\nabla\Qold\right):\pmb{\phi}\,dx+\int_\Omega s(\pmb{u}, \Qold):\pmb{\phi}\,dx+M\int_\Omega \pmb{H}:\pmb{\phi}\,dx\\
  &\quad\,\,+\frac{1}{\Delta t}\int_\Omega \pmb{H}:\testvar\,dx+\frac{L}{\Delta t}\int_\Omega \nabla\pmb{Q}:\nabla\testvar\,dx+\frac{1}{\Delta t}\int_\Omega \left(\Pold:\pmb{Q}\right)\left(\Pold:\testvar\right)\,dx\coloneqq \sum_{k=1}^{12} A_k^{n+1},
  \end{aligned}
\end{equation}
    and the right-hand side is
    \begin{equation}
        \label{eq:f_n_in_iteration}
        f_n\left(\begin{pmatrix}
            \pmb{\psi}, \pmb{\varphi}, \pmb{\phi}
        \end{pmatrix}
        \right)=\frac{1}{\Delta t}\left\langle \uold, \pmb{\psi} \right\rangle-\left\langle \nabla p^n,\pmb{\psi}\right\rangle+\frac{1}{\Delta t}\left\langle \Qold, \pmb{\phi} \right\rangle+\frac{1}{\Delta t}\left\langle \pmb{F}^n, \pmb{\varphi} \right\rangle.
    \end{equation}
    \end{subequations}
    From Lax-Milgram theorem\cite{evans10}, we infer that it is enough to show that $a_{n+1}$ is bounded and coercive. We will start with the boundedness. Given $(\uold,\Qold,\Hold)\in H^1_0(\Omega)\times \left(H^1_0(\Omega)\cap H^2(\Omega)\right)\times L^2(\Omega)$ and a fixed $\Delta t$, the terms $A_1^{n+1}, A_3^{n+1}, A_6^{n+1}, A_9^{n+1}, A_{10}^{n+1}, A_{11}^{n+1}$ can be bounded by Cauchy-Schwarz inequality as
    \begin{equation*}
        \label{eq:A_1}
       \lvert A_1^{n+1}\rvert \leq\frac{1}{\Delta t}\|\pmb{u}\|\,\|\testpsi\|\leq \frac{1}{\Delta t}\|\pmb{u}\|_{H^1_0}\,\|\testpsi\|_{H^1_0},
    \end{equation*}
    \begin{equation*}
        \label{eq:A_3}
       \lvert A_3^{n+1}\rvert \leq\mu\|\nabla\pmb{u}\|\,\|\nabla\testpsi\|\leq \mu\|\pmb{u}\|_{H^1_0}\,\|\testpsi\|_{H^1_0},
    \end{equation*}
    \begin{equation*}
        \label{eq:A_6}
       \lvert A_6^{n+1}\rvert \leq\frac{1}{\Delta t}\|\pmb{Q}\|\,\|\testphi\|\leq \frac{1}{\Delta t}\|\pmb{Q}\|_{H^1_0}\,\|\testphi\|,
    \end{equation*}
     \begin{equation*}
        \label{eq:A_9}
       \lvert A_9^{n+1}\rvert \leq M\|\pmb{H}\|\,\|\testphi\|,
    \end{equation*}
     \begin{equation*}
        \label{eq:A_10}
       \lvert A_{10}^{n+1}\rvert\leq \frac{1}{\Delta t}\|\pmb{H}\|\,\|\testvar\|\leq \frac{1}{\Delta t}\|\pmb{H}\|\,\|\testvar\|_{H^1_0},
    \end{equation*}
    \begin{equation*}
        \label{eq:A_11}
       \lvert A_{11}^{n+1}\rvert \leq\frac{L}{\Delta t}\|\nabla\pmb{Q}\|\,\|\nabla\testvar\|\leq \frac{L}{\Delta t}\|\pmb{Q}\|_{H^1_0}\,\|\testvar\|_{H^1_0}.
    \end{equation*}
    Using the H\"older inequality and the Sobolev inequality, we can estimate $A_2^{n+1}$ as
\begin{equation*}
    \label{eq:A_2}
    \lvert A_2^{n+1}\rvert\leq \|\uold\|_{L^4}\,\|\nabla \pmb{u}\|\,\|\testpsi\|_{L^4}\leq C\|\uold\|_{H^1_0}\,\|\pmb{u}\|_{H^1_0}\,\|\testpsi\|_{H^1_0}\leq C \|\pmb{u}\|_{H^1_0}\,\|\testpsi\|_{H^1_0}.
\end{equation*}
Similar tricks can be applied to control $A_5^k$ and $A_7^k$. Specifically, we have
\begin{equation*}
    \label{eq:A_5}
    \lvert A_5^{n+1}\rvert \leq \|\pmb{H}\|\,\|\nabla \Qold\|_{L^4}\,\|\testpsi\|_{L^4}\leq C\|\pmb{H}\|\,\|\Qold\|_{H^2}\,\|\testpsi\|_{H^1_0}\leq C \|\pmb{H}\|\,\|\testpsi\|_{H^1_0},
\end{equation*}
and
\begin{equation*}
    \label{eq:A_7}
     \lvert A_7^{n+1}\rvert \leq \|\pmb{u}\|_{L^4}\,\|\nabla \Qold\|_{L^4}\,\|\testphi\|\leq C\|\pmb{u}\|_{H^1_0}\,\|\Qold\|_{H^2}\,\|\testphi\|\leq C \|\pmb{u}\|_{H^1_0}\,\|\testphi\|.
\end{equation*}
Thanks to Lemma~\ref{lem:agmon} and Lemma~\ref{lem:P_lipschitz}, we obtain
\begin{equation*}
\label{eq:A_12}
\begin{aligned}
    \lvert A_{12}^{n+1}\rvert \leq \frac{1}{\Delta t}\|\Pold\|^2_{L^\infty}\,\|\pmb{Q}\|\,\|\pmb{\varphi}\|&\leq \frac{\Tilde{L}^2}{\Delta t}\|\Qold\|_{L^\infty}^2\|\pmb{Q}\|_{H^1_0}\,\|\testvar\|_{H^1_0}\\
    &\leq C\|\Qold\|_{H^1_0}\,\|\Qold\|_{H^2}\,\|\pmb{Q}\|_{H^1_0}\,\|\testvar\|_{H^1_0}\leq C \|\pmb{Q}\|_{H^1_0}\,\|\testvar\|_{H^1_0}.
\end{aligned}
\end{equation*}
Recalling definition \eqref{eq:S}, \eqref{eq:numerical_s} and \eqref{eq:sigma} and using Lemma~\ref{lem:agmon}, we can estimate the remaining two terms $A_4^{n+1}, A_8^{n+1}$ as
\begin{equation*}
    \label{eq:A_4}
    \begin{aligned}
        \lvert A_4^{n+1}\rvert&=\left\lvert  \int_\Omega \left[   \Qold\pmb{H}-\pmb{H}\Qold-\xi\left(\pmb{H}\Qold+\Qold\pmb{H} \right)-\frac{2\xi}{d}\,\pmb{H}+2\xi(\Qold :\pmb{H})\Qold \right]:\nabla\testpsi\,dx \right\rvert\\
        &\leq C\,\left( \|\Qold\|_{L^\infty}+\|\Qold\|_{L^\infty}^2+1\right)\,\|\pmb{H}\|\,\|\nabla\testpsi\|\leq C \|\pmb{H}\|\,\|\testpsi\|_{H^1_0},
    \end{aligned}
\end{equation*}
    and
\begin{equation*}
    \label{eq:A_8}
    \begin{aligned}
        \lvert A_8^{n+1}\rvert &=\left\lvert  \int_\Omega \left[   \pmb{W}\Qold-\Qold\pmb{W}+\xi\left(\Qold
\pmb{D}+\pmb{D}\Qold\right)+\frac{2\xi}{d}\pmb{D}-2\xi(\pmb{D}:\Qold)
\left(\Qold+\frac{1}{d}\pmb{I}\right)-\frac{2\xi}{d^2}(\Div \pmb{u})\,\pmb{I} \right]:\testphi\,dx \right\rvert \\
&\leq C\,\left( \|\Qold\|_{L^\infty}+\|\Qold\|_{L^\infty}^2+1\right)\,\|\nabla \pmb{u}\|\,\|\testphi\|\leq C\|\pmb{u}\|_{H^1_0}\,\|\testphi\|.
    \end{aligned}
\end{equation*}
Combining these estimates on $A_i, i=1, 2, \cdots, 12$, we conclude that
\begin{equation}
  \label{eq:bilinear_a_boundness}
\begin{aligned}
     \left \lvert a_{n+1}\left(\begin{pmatrix}
    \pmb{u}, \pmb{Q},  \pmb{H}
\end{pmatrix},\begin{pmatrix}
    \pmb{\psi}, \testvar, \pmb{\phi}\end{pmatrix}\right)\right\rvert&\leq C\left(\|\pmb{u}\|_{H^1_0}+\|\pmb{Q}\|_{H^1_0}+\|\pmb{H}\|\right)\,\left(\|\testpsi\|_{H^1_0}+\|\testvar\|_{H^1_0}+\|\testphi\| \right)\\
    &\leq C\,\left\|\begin{pmatrix}
    \pmb{u}, \pmb{Q},  \pmb{H}
\end{pmatrix}\right\|_{H^1_0(\Omega)\times H^1_0(\Omega)\times L^2(\Omega)}\,\left\|\begin{pmatrix}
    \testpsi, \testvar,  \testphi
\end{pmatrix}\right\|_{H^1_0(\Omega)\times H^1_0(\Omega)\times L^2(\Omega)},
\end{aligned}
\end{equation}
    which completes the proof of the boundedness of thebilinear form $a_{n+1}$.
    
    Next we show the coercivity of $a_{n+1}$. To do so, we choose $\begin{pmatrix}
     \pmb{\psi},   \testvar, \testphi     \end{pmatrix}=\begin{pmatrix}
        \pmb{u}, \pmb{Q}, \pmb{H}     \end{pmatrix}$, it follows from Lemma \ref{lem:properties_B}, Lemma \ref{lem:S_Sigma_cancellation_numerical} that
    
      \begin{equation}
        \label{eq:a_bilinear_coercivity}
        \begin{aligned}
           &\quad\,\, a\left( \begin{pmatrix}
    \pmb{u},
    \pmb{Q}, \pmb{H}
\end{pmatrix},\begin{pmatrix}
    \pmb{u}, 
    \pmb{Q}, \pmb{H}\end{pmatrix}\right)\\
    &=  \frac{1}{\Delta t}\int_\dom \pmb{u}\cdot\pmb{u}\,dx+\Tilde{B}(\uold, \pmb{u}, \pmb{u})+\mu\int_\dom \nabla\pmb{u}\cdot\nabla\pmb{u}\,dx+\int_\dom\Sigma(\Qold, \pmb{H}):\nabla\pmb{u}\,dx+\int_\Omega \pmb{H}\nabla\Qold\cdot\pmb{u}\,dx\\
  &\quad\,\,-\frac{1}{\Delta t}\int_\Omega \pmb{Q}:\pmb{H}\,dx-\int_\Omega (\pmb{u}\cdot\nabla \Qold):\pmb{H}\,dx +\int_\Omega s(\pmb{u},\Qold):\pmb{H}\,dx+M\int_\Omega \pmb{H}:\pmb{H}\,dx\\
  &\quad\,\,+\frac{1}{\Delta t}\int_\Omega \pmb{H}:\pmb{Q}\,dx +\frac{L}{\Delta t}\int_\Omega \nabla\pmb{Q}:\nabla\pmb{Q}\,dx +\frac{1}{\Delta t}\int_\Omega (\Pold:\pmb{Q})^2\,dx\\
  &=\frac{1}{\Delta t}\|\pmb{u}\|^2+\mu\|\nabla\pmb{u}\|^2+M\|\pmb{H}\|^2+\frac{L}{\Delta t}\|\nabla\pmb{Q}\|^2+\frac{1}{\Delta t}\|\Pold:\pmb{Q}\|^2\\
  &\geq C\left(\|\pmb{u}\|_{H^1_0}^2+\|\pmb{Q}\|_{H^1_0}^2+\|\pmb{H}\|^2\right),
  \end{aligned}
    \end{equation}
    for some constant $C>0$ which depends on $\mu, M$, and $\Delta t$. Thus, given $(\uold, \Qold, \Hold,p^n,r^n)\in H^1_0(\Omega)\times\left( H^1_0(\Omega)\cap H^2(\Omega)\right)\times L^2(\Omega)\times L^2(\Omega)\times L^2(\Omega)$, there exists a unique $(\uint, \Qnew, \Hnew)\in H^1_0(\Omega)\times H^1_0(\Omega)\times L^2(\Omega)$ solving~\eqref{eq:uQevolution}. Then standard results about elliptic equations~\cite{evans10} lift the regularity of $\Qnew$ to $H^2(\Omega)$ due to \eqref{eq:Hn+1_rewrite}.
    
    As it is stated in Remark \ref{rem:Step2_Helmholtz_decomposition}, the uniqueness and existence of $\unew$ and $p^{n+1}$ are guaranteed by the Helmholtz decomposition. Using \eqref{eq:scheme_projection} and \eqref{eq:scheme_divergencefree}, for any smooth function $\pmb{\psi}$ with compact support in $[0,T)\times\Omega$, $p^{n+1}$ solves
    \begin{equation}
        \label{eq:pressure_laplace}
        \begin{aligned}
             \int_\Omega \nabla p^{n+1}\cdot \nabla\pmb{\psi}\,dx&=\int_\Omega \nabla p^{n}\cdot \nabla\pmb{\psi}\,dx-\frac{1}{2\Delta t}\int_\Omega(\unew-\uint)\cdot\nabla\pmb{\psi}\\
             &= \int_\Omega\nabla p^n\cdot\nabla\pmb{\psi}-\frac{1}{2\Delta}\int_\Omega(\nabla\cdot\uint)\,\pmb{\psi}\,dx,
        \end{aligned}
    \end{equation}
    which implies that $p^{n+1}\in H^2(\Omega)$ given $p^n\in H^2(\Omega)$. In addition, it follows from \eqref{eq:scheme_projection} that $\unew=\uint-2(\nabla p^{n+1}-\nabla p^n)\Delta t\in H^1(\Omega)$.
    
    We have shown:
\begin{theorem}
    \label{thm:Solvability_scheme}
    Given the initial value $(\pmb{u}^0,\pmb{Q}^0)\in H^1_0(\Omega)\times \left(H^1_0(\Omega)\cap H^2(\Omega)\right)$, the numerical scheme \eqref{eq:uQevolution}-\eqref{eq:pressureupdate} can be solved iteratively with $(\pmb{u}^n,\pmb{Q}^n)\in H^1(\Omega)\times \left(H^1_0(\Omega)\cap H^2(\Omega)\right)$ for every $n\in\mathbb{Z}. $
\end{theorem}
\subsection{Energy stability}
\begin{lemma}
    \label{lem:basic_discrete_energy_law}
    The numerical scheme \eqref{eq:uQevolution}-\eqref{eq:pressureupdate} is unconditionally energy stable and satisfies the semi-discrete energy dissipation law
    \begin{equation}
        \label{eq:basic_semi-discrete_energy_law}
        \begin{aligned}
            E^{N+1}&+\frac{1}{4}\sum_{n=0}^{N-1}\|\unew-\uint\|^2
     +\frac{1}{2}\sum_{n=0}^N\|\uint-\uold\|^2+\frac{L}{2}\sum_{n=0}^N\|\nabla \Qnew-\nabla\Qold\|^2\\
     &\quad\quad\quad\quad+\frac{1}{2}\sum_{n=0}^N\|r^{n+1}-r^n\|^2
            +\mu \sum_{n=0}^N\|\nabla \uint\|^2\,\Delta t+M\sum_{n=0}^N\|\Hnew\|^2\,\Delta t =E^0,
        \end{aligned}
    \end{equation}
    for all integers $N\in\left[0,\floor{\frac{T}{\Delta t}}\right]$ where
    \begin{equation*}
        \label{eq:discrete_energy_N}
        \begin{aligned}
             E^{N+1}&=\frac{1}{2}\|\Tilde{\pmb{u}}^{N+1}\|^2+\frac{1}{2}\|\uN\|^2
             +\frac{L}{2}\|\nabla \QN\|^2+\frac{1}{2}\|r^{N+1}\|^2+\|\nabla p^{N+1}\|^2\Delta t^2.
        \end{aligned}
    \end{equation*}
\end{lemma}

\begin{proof}
According to Theorem \ref{thm:Solvability_scheme}, $\uint\in H^1_0(\Omega)$ for every $n\geq0$. It allows us to choose $\testvar=\uint\Delta t$ as a test function in \eqref{eq:Dtu_step1} to get
\begin{equation*}
    \label{eq:energy_u}
    \begin{aligned}
    \frac{1}{2}&\left( \|\uint\|^2-\| \uold\|^2+\|\uint-\uold\|^2 \right)\\
    &=-\langle \nabla p^n, \uint\rangle\Delta t-\mu\|\nabla\uint\|^2\Delta t-\langle \sigmanew, \nabla\uint\rangle\Delta t-\langle \Hnew\nabla \Qold, \uint\rangle\Delta t,
    \end{aligned}
\end{equation*}
where we have used the fact that $\Tilde{B}(\uold, \uint, \uint)=0$ by Lemma \ref{lem:properties_B}. 
Taking $-\Delta t \Hnew$ as a test function in \eqref{eq:DtQ_step1} and $\Qnew-\Qold$ in \eqref{eq:Hn+1}, adding the two equations, and using \eqref{eq:Dtr_step1},  we have
\begin{equation*}
\label{eq:energy_Q}
    \begin{aligned}
    \frac{L}{2}\big(\|\nabla \Qnew\|^2&-\|\nabla \Qold\|^2+\|\nabla \Qnew-\nabla \Qold\|^2\big) + \frac{1}{2}\left(\|\rint\|^2-\|r^n\|^2+\|\rint-r^n\|^2\right)\\
    &=-M\|\Hnew\|^2\Delta t+\langle \uint\cdot\nabla \Qold, \Hnew\rangle\Delta t -\langle \Snew, \Hnew\rangle\Delta t.
    \end{aligned}
\end{equation*}
Lemma \ref{lem:S_Sigma_cancellation_numerical} implies that
\begin{equation*}
    \label{eq:Sigma_S_cancel}
    \langle \nabla\uint, \,\sigmanew\rangle+\langle \Hnew, \Snew\rangle=0
\end{equation*}
Taking the inner product of \eqref{eq:scheme_projection} with $\frac{1}{2}\,\unew\Delta t$ and $\frac{1}{4}\,(\unew+\uint)\,\Delta t$ respectively, and using the divergence free condition \eqref{eq:scheme_divergencefree} for $\unew$, we have
\begin{subequations}
\label{eq:energy_projection}
\begin{equation}
    \label{eq:energy_projection1}
    \frac{1}{4}\|\unew\|^2-\frac{1}{4}\|\uint\|^2+\frac{1}{4}\|\unew-\uint\|^2=0,
\end{equation}
\begin{equation}
    \label{eq:energy_projection2}
    \frac{1}{4}\|\unew\|^2-\frac{1}{4}\|\uint\|^2=-\frac{1}{2}\,\langle \nabla p^{n+1}-\nabla p^n, \uint\rangle\,\Delta t.
\end{equation}
\end{subequations}
Adding up these estimates together, we obtain
\begin{equation*}
    \label{eq:energy_induction}
    \begin{aligned}
   &\quad\,\,\left( \frac{1}{2}\|\unew\|^2+\frac{L}{2}\|\nabla \Qnew\|^2+\frac{1}{2}\|r^{n+1}\|^2\right)- \left( \frac{1}{2}\|u^{n}\|^2+\frac{L}{2}\|\nabla Q^{n}\|^2+\frac{1}{2}\|r^{n}\|^2\right)\\
   &\quad\quad\quad+\frac{1}{4}\|\unew-\uint\|^2
   +\frac{1}{2}\|\uint-\uold\|^2+\frac{L}{2}\|\nabla \Qnew-\nabla \Qold\|^2+\frac{1}{2}\|r^{n+1}-r^n\|^2\\
   &=-\frac{1}{2}\langle \nabla p^{n+1}+\nabla p^n, \uint\rangle\,\Delta t
   - \mu\|\nabla \uint\|^2\Delta t-M\|\Hnew\|^2\Delta t\\
   &=-\frac{1}{2}\langle \nabla p^{n+1}+\nabla p^n,\, 2\left(\nabla p^{n+1}-\nabla p^n\right)\,\Delta t \rangle\,\Delta t
   -\mu\|\nabla \uint\|^2\Delta t-M\|\Hnew\|^2\Delta t\\
   &=-\left(\|\nabla p^{n+1}\|^2-\|\nabla p^n\|^2 \right)\,\Delta t^2
   -\mu\|\nabla \uint\|^2\Delta t-M\|\Hnew\|^2\Delta t.
    \end{aligned}
\end{equation*}
Summing up it from $n=0$ to $N$, we yield
\begin{equation*}
    \label{eq:energy_N}
    \begin{aligned}
       &\frac{1}{2}\|\uN\|^2+\frac{L}{2}\|\nabla \QN\|^2+\frac{1}{2}\|r^{n+1}\|^2+\|\nabla p^{N+1}\|^2\Delta t^2\\
       &\quad+\frac{1}{4}\sum_{n=0}^N\|\unew-\uint\|^2
     + \frac{1}{2}\sum_{n=0}^N\|\uint-\uold\|^2+\frac{L}{2}\sum_{n=0}^N\|\nabla\Qnew-\nabla\Qold\|^2+\frac{1}{2}\sum_{n=0}^N\|r^{n+1}-r^n\|^2\\
       &\quad\quad+\mu\,\sum_{n=0}^N\|\nabla \uint\|^2\,\Delta t+M\sum_{n=0}^N\|\Hnew\|^2\,\Delta t
       =\frac{1}{2}\|\pmb{u}^0\|^2+\frac{L}{2}\|\nabla \pmb{Q}^0\|^2+\frac{1}{2}\|r^{0}\|^2+\|\nabla p^{0}\|^2\Delta t^2.
    \end{aligned}
\end{equation*}
Using~\eqref{eq:energy_projection1} once more, we can also rewrite this equation as
%
%
\begin{equation}
\label{eq:energy_N_2}
\begin{aligned}
   &\quad\frac{1}{4}\|\Tilde{\pmb{u}}^{N+1}\|^2+\frac{1}{4}\|\uN\|^2+\frac{L}{2}\|\nabla \QN\|^2+\frac{1}{2}\|r^{n+1}\|^2+\|\nabla p^{N+1}\|^2\Delta t^2
   \\
   &\quad\quad+\frac{1}{4}\sum_{n=0}^{N-1}\|\unew-\uint\|^2
   +\frac{1}{2}\sum_{n=0}^N\|\uint-\uold\|^2+\frac{L}{2}\sum_{n=0}^N\|\nabla \Qnew-\nabla\Qold\|^2\\
   &\quad\quad+\frac{1}{2}\sum_{n=0}^N\|r^{n+1}-r^n\|^2+\mu\,\sum_{n=0}^N\|\nabla \uint\|^2\,\Delta t+M\sum_{n=0}^N\|\Hnew\|^2\,\Delta t\\
   &\quad\quad=\frac{1}{2}\|\pmb{u}^0\|^2+\frac{L}{2}\|\nabla \pmb{Q}^0\|^2+\frac{1}{2}\|r^{0}\|^2+\|\nabla p^{0}\|^2\Delta t^2.
\end{aligned}
\end{equation}
This concludes the proof of the discrete energy law of the system. 
\end{proof}

Next, we define piece-wise linear in time interpolations based on the approximants 
$(\uold, \Qold, p^n, r^n)$, $1\leq n\leq \floor{\frac{T}{\Delta t}}$. Specifically, given $\Delta t>0$, we define  $(\usol, \uappr, \Qsol, \rsol)$ as piece-wise linear interpolation of $\uold, \uint, \Qold, r^n$, that is,
\begin{subequations}
\label{eq:numerical_interpolation_solutions}
\begin{equation}
    \label{eq:usol}
    \usol(t) = \sum_{n=0}^{N-1}  \left[\frac{(n+1)\Delta t-t}{\Delta t}\,\uold+\frac{t-n\Delta t}{\Delta t}\,\unew\right]\,\chi_{S_n},
\end{equation}
\begin{equation}
    \label{eq:uappr}
    \uappr(t) = \sum_{n=0}^{N-1}  \left[\frac{(n+1)\Delta t-t}{\Delta t}\,\Tilde{\pmb{u}}^{n}+\frac{t-n\Delta t}{\Delta t}\,\Tilde{\pmb{u}}^{n+1}\right]\,\chi_{S_n},
\end{equation}
\begin{equation}
    \label{eq:Qsol}
    \Qsol(t) = \sum_{n=0}^{N-1}  \left[\frac{(n+1)\Delta t-t}{\Delta t}\,\Qold+\frac{t-n\Delta t}{\Delta t}\,\Qnew\right]\,\chi_{S_n},
\end{equation}
\begin{equation}
    \label{eq:Qbacksol}
    \Qbacksol(t) = \sum_{n=0}^{N-1}  \left[\frac{(n+1)\Delta t-t}{\Delta t}\,\pmb{Q}^{n-1}+\frac{t-n\Delta t}{\Delta t}\,\Qold\right]\,\chi_{S_n},
\end{equation}
\begin{equation}
    \label{eq:Psol}
    \Psol(t) = \sum_{n=0}^{N-1}  \left[\frac{(n+1)\Delta t-t}{\Delta t}\,\pmb{P}^{n-1}+\frac{t-n\Delta t}{\Delta t}\,\Pold\right]\,\chi_{S_n},
\end{equation}
\begin{equation}
    \label{eq:Hsol}
    \Hsol(t) = \sum_{n=0}^{N-1}  \left[\frac{(n+1)\Delta t-t}{\Delta t}\,\Hold+\frac{t-n\Delta t}{\Delta t}\,\Hnew\right]\,\chi_{S_n},
\end{equation}
\begin{equation}
    \label{eq:rsol}
    \rsol(t) = \sum_{n=0}^{N-1}  \left[\frac{(n+1)\Delta t-t}{\Delta t}\,r^n+\frac{t-n\Delta t}{\Delta t}\,r^{n+1} \right]\,\chi_{S_n},
\end{equation}
\end{subequations}
where $S_n=[n\Delta t, (n+1)\Delta t)$ and $\chi_{S_n}$ is the characteristic function on $S_n$. Our goal is to use the Aubin-Lions lemma, Lemma~\ref{lem:Aubin-Lions} to deduce pre-compactness of these interpolants. To be able to do so, we need to derive uniform (in $\Delta t$) estimates on their time derivatives. We summarize these estimates for regularity in time in the following two lemmas. The first one states the regularity for time derivative of velocity field $\usol$.
\begin{lemma}
\label{lem:u_H^-1_estimate} 
Let $V=H^2(\Omega)\cap H^1_{0,\sigma}(\Omega)$. For every $\Delta t>0$, we have
\begin{equation*}
    \partial_t \usol\in L^{2}(0,T;V').
\end{equation*}
\end{lemma}
\begin{proof}
From \eqref{eq:Dtu_step1}, we infer that for any $\pmb{\phi}\in L^{2}(0,T;V)  $,
\begin{equation*}
\begin{aligned}
    \left\langle \frac{\uint-\uold}{\Delta t}, \pmb{\phi}(\cdot, t)\right\rangle=-\langle B(\uold, \uint), \pmb{\phi}(\cdot, t)\rangle&- \langle \mu \nabla \uint, \nabla\pmb{\phi}(\cdot, t)\rangle\\
    &-\langle \sigmanew, \nabla\pmb{\phi}(\cdot, t)\rangle-\langle \Hnew\nabla \Qold, \pmb{\phi}(\cdot, t)\rangle\coloneqq\sum_{k=1}^4 I_k^n.
\end{aligned}
    \label{eq:Dtu_H^-1}
\end{equation*}
To derive the regularity estimate, we will control $I_1^n$ to $I_4^n$ separately. Using integration by parts and the energy estimate \eqref{eq:energy_N_2},
we obtain
\begin{align*}
   \left|\sum_{n=0}^{N-1} \int_{n\Delta t}^{(n+1)\Delta t} I_1^n\,dt\right|&=\left\lvert\sum_{n=0}^{N-1} \int_{n\Delta t}^{(n+1)\Delta t}\int_\Omega(\uold\cdot\nabla)\uint\cdot\pmb{\phi}\,dx dt\right\rvert \\
       &\leq \sum_{n=0}^{N-1} \int_{n\Delta t}^{(n+1)\Delta t}\|\pmb{\phi}(\cdot, t)\|_{L^\infty}\|\uold\|\|\nabla\uint\|\,dt \\
     &\stackrel{\eqref{eq:energy_N_2}}{\leq} C \sum_{n=0}^{N-1} \int_{n\Delta t}^{(n+1)\Delta t}\|\pmb{\phi}(\cdot, t)\|_{L^\infty}\|\nabla\uint\|\,dt\\
     &\leq C \sum_{n=0}^{N-1} \int_{n\Delta t}^{(n+1)\Delta t}\|\pmb{\phi}(\cdot, t)\|_{H^2}\|\nabla\uint\|\,dt\\
     &\leq C\sum_{n=0}^{N-1}\left( \int_{n\Delta t}^{(n+1)\Delta t}\|\nabla\uint\|^2\,dt \right)^{\frac{1}{2}}\,\left( \int_{n\Delta t}^{(n+1)\Delta t}\|\pmb{\phi}(\cdot,t)\|_{H^2}^2\,dt \right)^{\frac{1}{2}}\\
     &\leq C\,\left( \sum_{n=0}^{N-1} \int_{n\Delta t}^{(n+1)\Delta t}\|\nabla\uint\|^2\,dt \right)^{\frac{1}{2}}\,\left( \sum_{n=0}^{N-1} \int_{n\Delta t}^{(n+1)\Delta t}\|\pmb{\phi}(\cdot,t)\|_{H^2}^2\,dt \right)^{\frac{1}{2}}\\
     &=C \left(\sum_{n=0}^{N-1}\|\nabla\uint\|^2\,\Delta t \right)^{\frac{1}{2}}\,\|\pmb{\phi}\|_{L^2(0,T;H^2)}
      \stackrel{\eqref{eq:energy_N_2}}{\leq} C\|\pmb{\phi}\|_{L^2(0,T;V)}.
       \end{align*}
$I_2$ can be estimated as
\begin{align*}
      \left| \sum_{n=0}^{N-1} \int_{n\Delta t}^{(n+1)\Delta t} I_2^n\,dt\right|&=\left\lvert\sum_{n=0}^{N-1} \int_{n\Delta t}^{(n+1)\Delta t}\int_\Omega\mu\nabla\uint\cdot\nabla\pmb{\phi}\,dxdt\right\rvert\\
        &\leq C \left( \sum_{n=0}^{N-1}\|\nabla \uint\|^2\,\Delta t \right)^{\frac{1}{2}}  \|\nabla \pmb{\phi}\|_{L^2([0,T)\times\Omega)}\stackrel{\eqref{eq:energy_N_2}}{\leq} C\|\pmb{\phi}\|_{L^2(0,T;V)}
    \end{align*}
By Definitions \eqref{eq:sigma} and \eqref{eq:SnSigman}, H\"{o}lder's inequality, Poincar\'{e}'s inequality and the Sobolev inequality, we have
\begin{align*}
         &    \left|\sum_{n=0}^{N-1} \int_{n\Delta t}^{(n+1)\Delta t} I_3^n\,dt\right|\\
  &\leq\sum_{n=0}^{N-1}\bigg\lvert\int_{n\Delta t}^{(n+1)\Delta t}\bigg\langle \left(\Qold\Hnew-\Hnew\Qold\right)-\xi\left(\Hnew\Qold+\Qold\Hnew \right)\\
    &\hphantom{\leq\sum_{n=0}^{N-1}\bigg\lvert\int_{n\Delta t}^{(n+1)\Delta t}\bigg\langle \left(\Qold\Hnew-\Hnew\Qold\right)}-\frac{2\xi}{d}\Hnew+2\xi(\Qold:\Hnew)\Qold, \nabla \pmb{\phi}(\cdot, t)\bigg\rangle\,dt\bigg\rvert\\
    &\leq C\sum_{n=0}^{N-1}\int_{n\Delta t}^{(n+1)\Delta t}\bigg( \|\Qold\|_{L^3}\,\|\Hnew\|\,\|\nabla\pmb{\phi}(\cdot,t)\|_{L^6}\\
    &\hphantom{\leq C\sum_{n=0}^{N-1}\int_{n\Delta t}^{(n+1)\Delta t}\bigg( \|\Qold\|_{L^3}\,\|\Hnew\|}+\|\Hnew\|\,\|\nabla\pmb{\phi}(\cdot,t)\|+\|\Hnew\|\,\|\Qold\|_{L^6}^2\,\|\nabla\pmb{\phi}(\cdot,t)\|_{L^6}\bigg)\,dt  \\
    &\leq  C\sum_{n=0}^{N-1}\int_{n\Delta t}^{(n+1)\Delta t}\bigg( \|\Qold\|_{H^1}\,\|\Hnew\|\,\|\pmb{\phi}(\cdot,t)\|_{H^2}\\
    &\hphantom{\leq  C\sum_{n=0}^{N-1}\int_{n\Delta t}^{(n+1)\Delta t}\bigg( \|\Qold\|_{H^1}\,\|\Hnew\|}+\|\Hnew\|\,\|\pmb{\phi}(\cdot,t)\|_{H^1}+\|\Hnew\|\,\|\Qold\|_{H^1}^2\,\|\pmb{\phi}(\cdot,t)\|_{H^2}\bigg)\,dt \\
    &\stackrel{\eqref{eq:energy_N_2}}{\leq} C\sum_{n=0}^{N-1}\int_{n\Delta t}^{(n+1)\Delta t}\bigg( \|\Hnew\|\,\|\pmb{\phi}(\cdot,t)\|_{H^2}+\|\Hnew\|\,\|\pmb{\phi}(\cdot,t)\|_{H^1}+\|\Hnew\|\,\|\pmb{\phi}(\cdot,t)\|_{H^2}\bigg)\,dt \\
         &\leq C \left( \sum_{n=0}^{N-1}\| \Hnew\|^2\,\Delta t \right)^{\frac{1}{2}}  \| \pmb{\phi}\|_{L^2(0,T; H^2)}\stackrel{\eqref{eq:energy_N_2}}{\leq} C\|\pmb{\phi}\|_{L^2(0,T;V)}
\end{align*}
To control $I_4$, we apply Lemma \ref{lem:agmon},
\begin{align*}
        \left|\sum_{n=0}^{N-1} \int_{n\Delta t}^{(n+1)\Delta t} I_4^n\,dt\right|&=\left\lvert\sum_{n=0}^{N-1}\int_{n\Delta t}^{(n+1)\Delta t}\int_\Omega \Hnew\nabla\Qold\cdot\pmb{\phi}\,dxdt\right\rvert\\
        &\leq \sum_{n=0}^{N-1}\int_{n\Delta t}^{(n+1)\Delta t}\|\pmb{\phi}(\cdot, t)\|_{L^\infty}\|\Hnew\|\,\|\nabla \Qold\|\,dxdt\\
        &\stackrel{\eqref{eq:energy_N_2}}{\leq}C \sum_{n=0}^{N-1}\int_{n\Delta t}^{(n+1)\Delta t}\|\pmb{\phi}(\cdot, t)\|_{L^\infty}\|\Hnew\|\,dxdt\\
        &\leq C \sum_{n=0}^{N-1}\int_{n\Delta t}^{(n+1)\Delta t}\|\pmb{\phi}(\cdot, t)\|_{H^2}\|\Hnew\|\,dxdt\\
        &\leq C\left( \sum_{n=0}^{N-1}\| \Hnew\|^2\,\Delta t \right)^{\frac{1}{2}}\left(\int_0^T\|\pmb{\phi}(\cdot,t)\|_{H^2}^2\,dt\right)^{\frac{1}{2}}\stackrel{\eqref{eq:energy_N_2}}{\leq} C\|\pmb{\phi}\|_{L^{2}(0,T;V)}
\end{align*}
According to scheme \eqref{eq:pressureupdate}, for each $n$, $\unew$ is the projection of $\uint$ onto the space of divergence free functions. Therefore, for every $\testphi\in L^2(0,T;V)$, we have
\begin{equation*}
\label{eq:Dtu_projection_estimate}
\int_{n\Delta t}^{(n+1)\Delta t}\left\langle \partial_t \usol, \testphi \right\rangle\,dt = \int_{n\Delta t}^{(n+1)\Delta t}\left\langle \frac{\unew-\uold}{\Delta t}, \testphi \right\rangle\,dt = \int_{n\Delta t}^{(n+1)\Delta t}\left\langle \Dtuone, \testphi \right\rangle\,dt.
\end{equation*}
Therefore, combining the estimates from $I_1$ to $I_4$, we have shown that
\begin{equation}
    \label{eq:Dtu_estimate}
    \begin{aligned}
        \|\partial_t \usol\|_{L^2(0,T;V')}&=\sup_{\testphi\in L^2(0,T;V)}\frac{\left\lvert \sum_{n=0}^{N-1}\int_{n\Delta t}^{(n+1)\Delta t}\left\langle \Dtuone, \testphi \right\rangle\,dt\right\rvert}{\|\testphi\|_{L^2(0,T;V)}}\leq C,
    \end{aligned}
\end{equation}
and so $\partial_t \usol\in L^{2}(0,T;V')$ uniformly in $\Delta t$.
\end{proof}

Next we show a uniform estimate in $\Delta t$ for $\partial_t\Qsol$.

\begin{lemma}
    \label{lem:Q_H^-1_estimate}
    For every $\Delta t>0$, we have
    \begin{equation*}
        \label{eq:partial_tQ_estimate}
        \partial_t\Qsol\in L^2(0,T; L^{\frac{6}{5}}).
    \end{equation*}
\end{lemma}
\begin{proof}
    For any function $\pmb{\varphi}\in L^2(0,T;L^6)$, 
    \begin{equation*}
        \label{eq:DtQ_H^-1}
         \left\langle  \frac{\Qnew-\Qold}{\Delta t}, \pmb{\varphi}\right\rangle =  -\langle \uint\cdot\nabla\Qold, \pmb{\varphi}\rangle + \langle\Snew, \pmb{\varphi}\rangle + M\langle \Hnew, \pmb{\varphi}\rangle\coloneqq \sum_{k=1}^3 J_k^n.
    \end{equation*}
    Using energy estimate \eqref{eq:energy_N_2}, the first term can be bounded as follows
    \begin{align*}
        \left|\sum_{n=0}^{N-1} \int_{n\Delta t}^{(n+1)\Delta t}J_1^n\,dt\right|&=\left\lvert\sum_{n=0}^{N-1} \int_{n\Delta t}^{(n+1)\Delta t}\int_\Omega \uint\cdot\nabla\Qold:\pmb{\varphi}\,dxdt  \right\rvert  \\
        &\leq \sum_{n=0}^{N-1} \int_{n\Delta t}^{(n+1)\Delta t} \|\uint\|_{L^4}\|\nabla\Qold\|\|\pmb{\varphi}(\cdot, t)\|_{L^4}\,dt\\
        &\stackrel{\eqref{eq:energy_N_2}}{\leq} C \sum_{n=0}^{N-1} \int_{n\Delta t}^{(n+1)\Delta t} \|\nabla\uint\|\|\pmb{\varphi}(\cdot, t)\|_{L^6}\,dt \\
        &\leq C \left(\sum_{n=0}^{N-1} \|\nabla\uint\|^2\,\Delta t \right)^{\frac{1}{2}}\left(\int_0^T \|\pmb{\varphi}(\cdot, t)\|_{L^6}^2\,dt \right)^{\frac{1}{2}} \stackrel{\eqref{eq:energy_N_2}}{\leq} C\|\pmb{\varphi}\|_{L^2(0,T;L^6)}
    \end{align*}
    Using the Sobolev inequality and definitions~\eqref{eq:S}, \eqref{eq:SnSigman}, we have
    \begin{align*}
            &\left|\sum_{n=0}^{N-1}\int_{n\Delta t}^{(n+1)\Delta t} J_2^n\,dt\right|
            \\&=\bigg\lvert \sum_{n=0}^{N-1}\int_{n\Delta t}^{(n+1)\Delta t} \Big\langle   \Tilde{\pmb{W}}^{n+1}\Qold-\Qold\Tilde{\pmb{W}}^{n+1}+\xi\left(\Qold
\Tilde{\pmb{D}}^{n+1}+\Tilde{\pmb{D}}^{n+1}\Qold\right)\\
    & \hphantom{=\bigg\lvert \sum_{n=0}^{N-1}\int_{n\Delta t}^{(n+1)\Delta t} \Big\langle }       +\frac{2\xi}{d}\Tilde{\pmb{D}}^{n+1}-\frac{2\xi}{d^2}\Div\uint \pmb{I}-2\xi(\Tilde{\pmb{D}}^{n+1}:\Qold)
    \left(\Qold+\frac{1}{d}\pmb{I}\right), \pmb{\varphi} \Big\rangle dt \bigg\rvert\\
    &\leq \sum_{n=0}^{N-1}\int_{n\Delta t}^{(n+1)\Delta t} \left(\|\nabla\uint\|\|\Qold\|_{L^4}\|\pmb{\varphi}(\cdot,t)\|_{L^4}+\|\nabla\uint\|\|\pmb{\varphi}(\cdot,t)\|+ \|\nabla\uint\|\|\Qold\|_{L^6}^2\|\pmb{\varphi}(\cdot,t)\|_{L^6} \right)\,dt\\
    &\leq C\sum_{n=0}^{N-1}\int_{n\Delta t}^{(n+1)\Delta t} \Big(\|\nabla\uint\|\|\nabla\Qold\|\|\pmb{\varphi}(\cdot,t)\|_{L^6}+\|\nabla\uint\|\|\pmb{\varphi}(\cdot,t)\|_{L^6}\\
    &\hphantom{leq C\sum_{n=0}^{N-1}\int_{n\Delta t}^{(n+1)\Delta t} \Big(\|\nabla\uint\|\|\nabla\Qold\|\|\pmb{\varphi}(\cdot,t)\|_{L^6}+\|\nabla\uint\|}+ \|\nabla\uint\|\|\nabla\Qold\|^2\|\pmb{\varphi}(\cdot,t)\|_{L^6} \Big)\,dt\\
    &\stackrel{\eqref{eq:energy_N_2}}{\leq} C\sum_{n=0}^{N-1}\int_{n\Delta t}^{(n+1)\Delta t} \left(\|\nabla\uint\|\|\pmb{\varphi}(\cdot,t)\|_{L^6}+\|\nabla\uint\|\|\pmb{\varphi}(\cdot,t)\|_{L^6}+ \|\nabla\uint\|\|\pmb{\varphi}(\cdot,t)\|_{L^6} \right)\,dt\\
    &\leq C \left(\sum_{n=0}^{N-1} \|\nabla\uint\|^2\,\Delta t \right)^{\frac{1}{2}}\left(\int_0^T \|\pmb{\varphi}(\cdot, t)\|_{L^6}^2\,dt \right)^{\frac{1}{2}} \stackrel{\eqref{eq:energy_N_2}}{\leq} C\|\pmb{\varphi}\|_{L^2(0,T;L^6)}
    \end{align*}
    The last term $J_3^n$ satisfies
    \begin{align*}
            \left|\sum_{n=0}^{N-1}\int_{n\Delta t}^{(n+1)\Delta t} J_3^n\,dt\right|&=\left \lvert \sum_{n=0}^{N-1}\int_{n\Delta t}^{(n+1)\Delta t} \int_\Omega \Hnew:\pmb{\varphi}\,dxdt \right\rvert\\
            &\leq  \sum_{n=0}^{N-1}\int_{n\Delta t}^{(n+1)\Delta t} \|\Hnew\|\|\pmb{\varphi}(\cdot,t)\|\,dt\\
            &\leq  \sum_{n=0}^{N-1}\int_{n\Delta t}^{(n+1)\Delta t} \|\Hnew\|\|\pmb{\varphi}(\cdot,t)\|_{L^6}\,dt\\
             &\leq C \left(\sum_{n=0}^{N-1} \|\Hnew\|^2\,\Delta t \right)^{\frac{1}{2}}\left(\int_0^T \|\pmb{\varphi}(\cdot, t)\|_{L^6}^2\,dt \right)^{\frac{1}{2}} \stackrel{\eqref{eq:energy_N_2}}{\leq} C\|\pmb{\varphi}\|_{L^2(0,T;L^6)}
    \end{align*}
    Combining these estimates for $J_1, J_2$ and $J_3$, we have shown $\partial_t\Qsol\in L^2(0,T; L^{\frac{6}{5})}$.
\end{proof}
This estimate naturally leads to the following corollary:
\begin{corollary}
\label{cor:Dtr_estimate}
   We have
   \begin{equation*}
       \partial_t\rsol\in L^2(0,T; L^1).
   \end{equation*}
\end{corollary}
\begin{proof}
    We obtain from \eqref{eq:Dtr_step1} that
    \begin{equation*}
        \label{eq:Dtr_estimate}
        \begin{aligned}
\int_0^T\left(\int_\Omega\lvert \partial_t \rsol(x,t)dx \rvert\right)^2dt&= \sum_{n=0}^{N-1}\int_{n\Delta t}^{(n+1)\Delta t} \norm{\Pold:\partial_t\Qsol(\cdot, t)}_{L^1}^2\,dt\\
            &\leq \sum_{n=0}^{N-1}\int_{n\Delta t}^{(n+1)\Delta t} \norm{\Pold}_{L^6}^2:\norm{\partial_t\Qsol(\cdot, t)}_{L^{\frac{6}{5}}}^2 dt\\
            &\stackrel{\text{Lemma } \ref{lem:P_lipschitz}}{\leq} C\sum_{n=0}^{N-1}\int_{n\Delta t}^{(n+1)\Delta t} \norm{\Qold}_{L^6}^2:\norm{\partial_t\Qsol(\cdot, t)}_{L^{\frac{6}{5}}}^2 dt\\
            &\leq C\sum_{n=0}^{N-1}\int_{n\Delta t}^{(n+1)\Delta t} \norm{\nabla\Qold} ^2:\norm{\partial_t\Qsol(\cdot, t)}_{L^{\frac{6}{5}}}^2 dt\\
         &\stackrel{\eqref{eq:energy_N_2}}{\leq }C\sum_{n=0}^{N-1}\int_{n\Delta t}^{(n+1)\Delta t} \norm{\partial_t\Qsol(\cdot, t)}_{L^{\frac{6}{5}}}^2\,dt=C\|\partial_t\Qsol\|_{L^2(0,T; L^{6/5})}^2<\infty.
            \end{aligned}
            \end{equation*}
  \end{proof}

As in~\cite{Paicu_global_existence, paicu_energy_dissipation}, we will show that the approximations of $\pmb{Q}$ are uniformly bounded in $L^2([0,T];H^2)$. This is critical for obtaining weak solutions. In~\cite{Paicu_global_existence, paicu_energy_dissipation}, this result is obtained via Sobolev embeddings and using the integrability of the bulk potential term in the energy. Due to the reformulation with the auxiliary variable, the same integrability is not available for the auxiliary variable $r$ through the a priori energy estimate. However, it is possible to obtain the $L^2([0,T];H^2)$-regularity using Lemma~\ref{lem:agmon}:  
\begin{lemma}\label{lem:Delta_Q_estimate}
    If $\pmb{Q}^0\in H^2(\Omega)$, then
    \begin{equation}
        \label{eq:Q_H2_estimate}
        \Delta t\sum_{k=1}^N \|\Delta \pmb{Q}^k\|^2\leq C.
    \end{equation}
    \end{lemma}
    \begin{proof}
    As it is shown in Theorem \ref{thm:Solvability_scheme}, for each $k\in\mathbb{N}$, $\pmb{Q}^k\in H^2(\Omega)$. Therefore, we can integrate by parts in \eqref{eq:Hn+1} which leads to 
    \begin{equation*}
        \label{eq:H_scheme_integration_by_parts}
        \langle \pmb{H}^{k+1}, \pmb{\phi}\rangle = L\langle \Delta \pmb{Q}^{k+1},\pmb{\phi}\rangle-\langle r^{k+1}\pmb{P}^k,\pmb{\phi}\rangle,
    \end{equation*}
    for any smooth $\pmb{\phi}$ with compact support. By density of $C_c^\infty(\Omega)$ in $L^2(\Omega)$, we can use test functions in $L^2(\dom)$ and in particular, we can choose $\Delta\pmb{Q}^{k+1}$ as a test function to obtain
    \begin{equation*}
        \label{eq:Q_H2_estimate_weak2strong}
        L\langle \Delta \pmb{Q}^{k+1}, \Delta \pmb{Q}^{k+1}\rangle =\langle \pmb{H}^{k+1}, \Delta \pmb{Q}^{k+1}\rangle+ \langle r^{k+1}\pmb{P}^k, \Delta \pmb{Q}^{k+1}\rangle.
    \end{equation*}
    
   Using Lemma \ref{lem:basic_discrete_energy_law}, Lemma ~\ref{lem:agmon}, Lemma \ref{lem:laplace_estimate} and Lemma \ref{lem:P_lipschitz}, we have
    \begin{equation*}
        \label{eq:Laplace_Q_L2}
        \begin{aligned}
            L\|\Delta \pmb{Q}^{k+1}\|^2
            &\leq C\left(\|\pmb{H}^{k+1}\|^2+\|r^{k+1}P(\pmb{Q}^k)\|^2 \right)+\frac{L}{4}\|\Delta \pmb{Q}^{k+1}\|^2\\
       &\leq C\left(\|\pmb{H}^{k+1}\|^2+\|P(\pmb{Q}^k)\|_{L^\infty}^2\|r^{k+1}\|^2 \right)+\frac{L}{4}\|\Delta \pmb{Q}^{k+1}\|^2\\
      &\stackrel{\text{Lemma }\ref{lem:P_lipschitz}}{\leq} C\left(\|\pmb{H}^{k+1}\|^2+\|\pmb{Q}^k\|_{L^\infty}^2\|r^{k+1}\|^2 \right)+\frac{L}{4}\|\Delta \pmb{Q}^{k+1}\|^2\\
      &\stackrel{\text{Lemma }\ref{lem:agmon}}{\leq} C\left(\|\pmb{H}^{k+1}\|^2+\|\pmb{Q}^k\|_{H^2}+1 \right)+\frac{L}{4}\|\Delta \pmb{Q}^{k+1}\|^2\\
      &\stackrel{\text{Lemma }\ref{lem:laplace_estimate}}{\leq} C\left(\|\pmb{H}^{k+1}\|^2+\|\Delta \pmb{Q}^k\|+1 \right)+\frac{L}{4}\|\Delta \pmb{Q}^{k+1}\|^2\\
      &\leq C\left(1+\|\pmb{H}^{k+1}\|^2\right)+\frac{L}{4}\|\Delta \pmb{Q}^k\|^2+\frac{L}{4}\|\Delta \pmb{Q}^{k+1}\|^2.
      \end{aligned}
    \end{equation*}
    Multiplying $\Delta t$ on both sides and summing from $k=0$ to $k=N-1$, we have
    \begin{equation*}
        \label{eq:Delta_Q_sum_estimate}
        \frac{L}{4}\|\Delta \pmb{Q}^N\|^2\Delta t+\frac{L }{2}\sum_{k=1}^N\|\Delta \pmb{Q}^k\|^2\Delta t\leq \frac{L}{4}\|\Delta \pmb{Q}^0\|^2\Delta t+\sum_{k=1}^N C\left(1+\|\pmb{H}^{k+1}\|^2\right)\Delta t,
    \end{equation*}
    which is bounded uniformly in $\Delta t$ thanks to the discrete energy estimate ~\eqref{eq:energy_N_2}.
    
    \end{proof}

\section{Convergence analysis}\label{sec:convergence}

In this section, we will prove convergence of the semi-discrete numerical scheme constructed in the previous section as the time step $\Delta t$ tends to zero. We will show that a subsequence of $\{\Qsol, \uappr, \pmb{H}_{\Delta t},\rsol\}_{\Delta t}$ converges to a weak solution of system \eqref{eq:ut}-\eqref{eq:H}. This leads to the following main theorem:
\begin{theorem}
    The piece-wise linear interpolations \eqref{eq:usol}-\eqref{eq:rsol} computed using scheme \eqref{eq:Dtu_step1}-\eqref{eq:scheme_divergencefree} converge up to
a subsequence to a weak solution of \eqref{eq:ut}-\eqref{eq:H} (as in Definition \ref{def:weakreformulation}) as $\Delta t\to0$.
\end{theorem}
\begin{proof}
Our proof utilizes the energy estimates derived in the last section for the linear interpolations defined in \eqref{eq:usol}-\eqref{eq:rsol}. Then we will use compactness theorems, such as Lemma \ref{lem:Aubin-Lions}, to extract a convergent subsequence and pass the limit, obtaining a weak solution of system \eqref{eq:ut}-\eqref{eq:H}.  We split the proof into several steps as follows.

\pmb{Step 1: Smoothing the initial data.  }
In order for Lemma~\ref{lem:Delta_Q_estimate} to be useful, we need $\Delta t \norm{\Delta \pmb{Q}^0}^2$ to be uniformly bounded in $\Delta t$. However, the initial data $\Qini$ may be less regular, for example, in $H^1(\dom)$ only. In order to approximate $\Qini$ with a sufficiently regular initial approximation $\pmb{Q}^0$, we proceed as follows: Given $\Qini\in H_0^1$, we determine $\pmb{Q}^0\in H^1_0\cap H^2$ by solving the equation
\begin{equation*}
    \label{eq:determine_Q0}
    (\pmb{I }-\Delta t\Delta)\,\pmb{Q}^0=\Qini.
\end{equation*}
We obtain from an energy estimate that
\begin{equation}
    \label{eq:Q0_approximation_energy_estimate}
    \|\nabla \pmb{Q}^0\|^2+\Delta t\|\Delta \pmb{Q}^0\|^2\leq \|\nabla \Qini\|\,\|\nabla \pmb{Q}^0\|\leq \frac{1}{2}\|\nabla \Qini\|^2+\frac{1}{2}\|\nabla \pmb{Q}^0\|^2.
\end{equation}
This implies that $\Delta t\|\Delta \pmb{Q}^0\|^2$ is bounded and, therefore, $\|\Delta t\Delta \pmb{Q}^0\|^2=O(\Delta t)$. As $\Delta t$ tends to $0$, $\|\Delta t\Delta \pmb{Q}^0\|^2$ tends to $0$. Then we can conclude that $\pmb{Q}^0\to \Qini$ strongly in $L^2$ and weakly in $H^1$.

\pmb{Step 2: Compactness.} The a priori estimates from the previous section can be summarized as follows: For any fixed $T>0$,
\begin{equation}
    \label{eq:estimates}
    \begin{aligned}
          &\sup_{\Delta t}\|\Qsol\|_{L^2([0,T];H^2)\cap L^\infty([0,T]; H^1)}< \infty,\quad\quad\sup_{\Delta t}\|\Qbacksol\|_{L^2([0,T];H^2)\cap L^\infty([0,T]; H^1)}< \infty,\\
          &\sup_{\Delta t}\|\usol\|_{L^2([0,T];H^1_\sigma)\cap L^\infty([0,T]; L^2_\sigma)}<\infty,\quad\quad \sup_{\Delta t}\|\uappr\|_{L^2([0,T];H^1)\cap L^\infty([0,T]; L^2)}<\infty\\
          & \hphantom{\sup_{\Delta t}\|\usol\|_{L^2([0,T];H^1_\sigma)\cap L^\infty([0,T]; L^2_\sigma)}}  \sup_{\Delta t}\|\rsol\|_{L^\infty([0,T]; L^2(\dom))}<\infty.
    \end{aligned}
\end{equation}
Lemma \ref{lem:u_H^-1_estimate} and Lemma \ref{lem:Q_H^-1_estimate} imply
\begin{equation}
    \label{eq:derivative_estimate}
  \sup_{\Delta t}\|\partial_t\usol\|_{L^2(0,T;V')}<\infty,\qquad \sup_{\Delta t}\|\partial_t\Qsol\|_{L^2(0,T;H^{-1})}<\infty.
\end{equation}
Noting that $L^2_{\sigma}(\Omega)$ is continuously embedded into $V'=\left(H^2(\Omega)\cap H^1_{0,\sigma}(\Omega)\right)'$, we can apply Lemma~\ref{lem:Aubin-Lions} to obtain that there exists $\pmb{u}\in L^2([0,T];H^1_{0,\sigma})\cap L^\infty([0,T]; L^2)$ and a subsequence of $\{\usol\}_{\Delta t}$, which will be denoted as $\{\usolsubseq\}_{m}$, such that
\begin{equation}
    \label{eq:u_sol_convergence}
    \usolsubseq\rightharpoonup \pmb{u} \text{ in } L^2(0,T;H^1_{0,\sigma}),\quad \usolsubseq\to \pmb{u} \text{ in } L^2(0,T; L_\sigma^2), \quad\usolsubseq(t)\stackrel{*}{\rightharpoonup} \pmb{u}(t) \text{ in } L^2 \text{ for a.e. } t\in[0,T],
\end{equation}
Similarly, for the Q-tensor, $H^1$ is continuously embedded into $H^{-1}(\Omega)$ and so we apply Lemma \ref{lem:Aubin-Lions} again, to obtain $\pmb{Q}, \pmb{Q}^*\in L^2([0,T]; H^2)\cap L^\infty([0,T]; H^1)$ and subsequences of $\{\Qsol\}_{\Delta t}$ and $\{\Qbacksol\}_{\Delta t}$ which will be denoted by $\Qsolsubseq$ and $\Qbacksolsubseq$, such that
\begin{equation}
    \label{eq:Qsol_convergence}
    \Qsolsubseq \rightharpoonup \pmb{Q} \text{ in } L^2(0,T; H^2),\quad \Qsolsubseq \to \pmb{Q} \text{ in } L^2(0,T; H^{1}),\quad \Qsolsubseq(t)\to \pmb{Q}(t) \text{ in } L^2,\, \forall\, t\in[0,T],
\end{equation}
\begin{equation}
    \label{eq:Qbacksol_convergence}
    \Qbacksolsubseq \rightharpoonup \pmb{Q}^* \text{ in } L^2([0,T]; H^2),\quad \Qbacksolsubseq \to \pmb{Q}^* \text{ in } L^2([0,T]; H^{1}),\quad \Qbacksolsubseq(t)\to \pmb{Q}^*(t) \text{ in } L^2, \,\forall\, t\in[0,T].
\end{equation}
According to Lemma \ref{lem:P_lipschitz}, the Lipschitz continuity of $P$ guarantees the strong convergence properties of subsequence $\{\Qbacksolsubseq\}_m$ hold as well for the sequence $\{P(\Qsolsubseq)\}_m$, that is,
\begin{equation}
    \label{eq:PQsol_convergence}
   \Psolsubseq \to P(\pmb{Q}^*) \text{ in } L^2([0,T)\times\dom).
\end{equation}
In view of the Banach-Alaoglu theorem~\cite{Folland} and Lemma~\ref{lem:Q_H^-1_estimate}, we can extract a weakly convergent subsequence $\{\partial_t \Qsolsubseq\}_m$ such that
\begin{equation}
    \label{eq:DtQ_convergence}
    \partial_t \Qsolsubseq \rightharpoonup \partial_t \pmb{Q} \text{ in } L^2([0,T];H^{-1}),
\end{equation}
and a weakly convergent subsequence of $\{\rsolsubseq\}_m$ from $\{\rsol\}_{\Delta t}$ such that
\begin{equation}
    \label{eq:r_convergence}
    \rsolsubseq \overset{\ast}{\rightharpoonup} r \text{ in } L^\infty([0,T]; L^2).
\end{equation}

\pmb{Step 3: Equivalence between $\pmb{Q}$ and $\pmb{Q}^*$ and convergence of $\uappr$.} This step's primary purpose is to show that the limit functions of the various subsequences coincide. Noting that $\Qsolsubseq$ differs from $\Qbacksolsubseq$ since they are interpolations of numerical solutions obtained at consecutive time steps, we can make use of the upper bound of the term $\sum_{n=0}^N\|\nabla \Qnew-\nabla\Qold\|^2$ obtained in Lemma \ref{lem:basic_discrete_energy_law} to deduce that
\begin{equation}
    \label{eq:Q-Q*}
    \begin{aligned}
        &\quad\,\,\|\pmb{Q}-\pmb{Q}^*\|_{L^2([0,T]; H^1)}\\
        &\leq \|\pmb{Q}-\Qsolsubseq\|_{L^2([0,T]; H^1)}+\|\Qsolsubseq-\Qbacksolsubseq\|_{L^2([0,T]; H^1)}+\|\pmb{Q}^*-\Qbacksolsubseq\|_{L^2([0,T]; H^1)}\\
        &=\|\pmb{Q}-\Qsolsubseq\|_{L^2([0,T]; H^1)}\\
        &\quad+\sum_{n=0}^{N}\left\|\frac{(n+1)\Delta t-t}{\Delta t}\,(\Qold-\pmb{Q}^{n-1})+\frac{t-n\Delta t}{\Delta t}\,(\Qnew-\Qold)\right\|_{L^2(S_n; H^1)}+\|\pmb{Q}^*-\Qbacksolsubseq\|_{L^2([0,T]; H^1)}\\
         &=\|\pmb{Q}-\Qsolsubseq\|_{L^2([0,T]; H^1)}\\
         &\quad+\sum_{n=0}^{N}\left[\left\|\Qold-\pmb{Q}^{n-1}\right\|_{H^1}+\left\|\Qnew-\Qold\right\|_{H^1}\right]\Delta t+\|\pmb{Q}^*-\Qbacksolsubseq\|_{L^2([0,T]; H^1)}\\
        &\leq \|\pmb{Q}-\Qsolsubseq\|_{L^2([0,T]; H^1)}+C\left(\sum_{n=0}^N\|\nabla \Qnew-\nabla \Qold\|^2\Delta t\right)^{\frac{1}{2}}
         +\|\pmb{Q}^*-\Qbacksolsubseq\|_{L^2([0,T]; H^1)}.
    \end{aligned}
 \end{equation}
As $\Delta t\to0$, the convergence results \eqref{eq:Qsol_convergence} and \eqref{eq:Qbacksol_convergence} imply that the first and third will go to $0$ as $\Delta t$ tends to $0$ while the second is $O(\sqrt{\Delta t})$ by energy estimate \eqref{eq:basic_semi-discrete_energy_law}, and so it goes to 0, too. So we conclude that $\pmb{Q}$ is equal to $\pmb{Q}^*$.

For the velocity field, though the sequence $\{\uappr\}_{m}$ does not preserve the divergence-free property on each step, we will show that the limit of its subsequence $\uapprsubseq$ agrees with $\pmb{u}$. To see this, we infer from definitions \eqref{eq:usol} and \eqref{eq:uappr} that
\begin{equation}
    \label{eq:u-u*}
    \begin{aligned}
   \|\pmb{u}-\uapprsubseq\|_{L^2([0,T)\times \dom)}
   &\leq \|\pmb{u}-\usolsubseq\|_{L^2([0,T)\times\Omega)}+\|\usolsubseq-\uapprsubseq\|_{L^2([0,T)\times\Omega)}\\
    &\leq \|\pmb{u}-\usolsubseq\|_{L^2([0,T)\times\Omega)}+C\left(\sum_{n=0}^{N}\|\uint-\unew\|^2\Delta t\right)^{\frac{1}{2}}.
    \end{aligned}
\end{equation}
As $\Delta t\to0$, the convergence result \eqref{eq:u_sol_convergence} implies that the first term will go to $0$ while the second is $O(\sqrt{\Delta t})$ by energy estimate \eqref{eq:basic_semi-discrete_energy_law}, and so it goes to $0$ as well. In this way, we have shown that $\uapprsubseq\to \pmb{u}$ strongly in $L^2\left([0,T)\times \Omega\right)$. For each $\Delta t$, we note that $\usol$ is divergence-free and therefore, by the weak convergence in \eqref{eq:u_sol_convergence}, we obtain that for almost every $t\in[0,T]$ and any smooth function $\phi\in C_c^\infty(\Omega)$, 
\begin{equation}
    \label{eq:usol_divergence_free}
    \begin{aligned}
   0= \lim_{m\to\infty}\int_\dom \phi(x)\Div \usolsubseq(t,x)\,dx&=-\lim_{m\to\infty}\int \nabla\phi(x)\cdot\usolsubseq(t,x)\,dx\\
   &=\int_\dom\nabla\phi(x)\cdot \pmb{u}(t,x)\,dx=\int_\dom\phi(x)\Div\pmb{u}(t,x)\,dx.
   \end{aligned}
\end{equation}
This implies that $\pmb{u}$ is weakly divergence-free which implies that it is divergence free almost everywhere in $[0,T)\times \dom$.

\pmb{Step 4: Weak convergence to $\pmb{H}$, $\pmb{S}$, $\pmb{\Sigma}$.} We let $\pmb{H}=L\Delta\pmb{Q}-rP(\pmb{Q})$. This is well-defined thanks to the regularity estimates we obtained for $\pmb{Q}$ in the previous steps. To obtain a representation of $\Hsol$ in terms of $r^n$ and $\pmb{P}^n$, we introduce the following piece-wise linear function $\widetilde{r\pmb{P}}_{\Delta t}$ to approximate $r\pmb{P}$, 
\begin{equation}
    \label{eq:rP_star}
   \widetilde{r\pmb{P}}_{\Delta t}=\sum_{n=0}^{N-1}  \left[\frac{(n+1)\Delta t-t}{\Delta t}\,r^n\pmb{P}^{n-1}+\frac{t-n\Delta t}{\Delta t}\,r^{n+1}\Pold \right]\,\chi_{S_n}.
\end{equation}
Recalling definitions \eqref{eq:Hsol} and \eqref{eq:Hn+1}, the interpolation $\Hsol$ satisfies the following weak form 
\begin{equation*}
 \label{eq:H_sol_weak_form}
   \left\langle \Hsol, \pmb{\phi}\right\rangle = -L\left\langle \nabla \Qsol,\nabla\pmb{\phi}\right\rangle-\left\langle \widetilde{r\pmb{P}}_{\Delta t}, \pmb{\phi} \right\rangle
\end{equation*}
for any smooth matrix-valued test function $\testphi$ with compact support in $[0,T)\times \Omega$. Accordingly, the subsequence $\Hsolsubseq$ satisfies
\begin{equation*}
    \label{eq:H_solsubseq_weak_form}
     \left\langle \Hsolsubseq, \pmb{\phi}\right\rangle = -L\left\langle \nabla \Qsolsubseq,\nabla\pmb{\phi}\right\rangle-\left\langle \widetilde{r\pmb{P}}_{\Delta t_m}, \pmb{\phi} \right\rangle.
\end{equation*}
To show that $\widetilde{r\pmb{P}}_{\Delta t_m}$ converges weakly to $rP(\pmb{Q})$, we introduce a piece-wise constant interpolation $\pmb{P}^*_{\Delta t}$ to approximate $\pmb{P}$ as
\begin{equation}
    \label{eq:P_star}
    \pmb{P}^*_{\Delta t}=\sum_{n=0}^{N-1}  \pmb{P}^{n-1}\,\chi_{S_n}.
\end{equation}
Then for any smooth test function $\testphi$, we have
\begin{equation*}
    \label{eq:rp_approximation}
    \begin{aligned}
&\int_0^T\int_\Omega\left(\widetilde{r\pmb{P}}_{\Delta t_m}-rP(\pmb{Q})\right):\testphi\,dxdt\\
        &= \underbrace{\int_0^T\int_\Omega \left(\widetilde{r\pmb{P}}_{\Delta t_m}-\rsolsubseq \pmb{P}^*_{\Delta t_m}\right):\testphi\,dxdt}_{K_1}+ \underbrace{\int_0^T\int_\Omega \rsolsubseq \left(\pmb{P}^*_{\Delta t_m}- \Psolsubseq\right):\testphi\,dxdt}_{K_2}\\
        &\quad\,\,+ \underbrace{\int_0^T\int_\Omega \rsolsubseq \left( \Psolsubseq-P(\pmb{Q}) \right):\testphi\,dxdt}_{K_3}+\underbrace{\int_0^T\int_\Omega (\rsolsubseq-r)P(\pmb{Q}):\testphi\,dxdt}_{K_4} 
    \end{aligned}
\end{equation*}
Our goal is to show that when $m\to\infty$, each $K_i, i=1, 2, 3, 4$ tends to $0$. With \eqref{eq:rsol}, \eqref{eq:rP_star}, \eqref{eq:P_star} and using Lemma \ref{lem:P_lipschitz}, Lemma \ref{lem:basic_discrete_energy_law}, we can estimate $K_1$ as
\begin{equation*}
    \label{eq:K_1}
    \begin{aligned}
       \lvert K_1\rvert&=\Bigg\lvert\int_0^T\int_\Omega \sum_{n=0}^{N-1}  \left[\frac{(n+1)\Delta t_m-t}{\Delta t_m}\,r^n\pmb{P}^{n-1}+\frac{t-n\Delta t_m}{\Delta t_m}\,r^{n+1}\Pold \right]\,\chi_{S_n}:\testphi\,dxdt\\
        &\quad\,\,-\int_0^T\int_\Omega \sum_{n=0}^{N-1}  \left[\frac{(n+1)\Delta t_m-t}{\Delta t_m}\,r^n\pmb{P}^{n-1}+\frac{t-n\Delta t_m}{\Delta t_m}\,r^{n+1}\pmb{P}^{n-1} \right]\,\chi_{S_n}:\testphi\,dxdt\Bigg\rvert\\
        &=\left\lvert\int_0^T\int_\Omega\sum_{n=0}^{N-1} \frac{t-n\Delta t_m}{\Delta t_m}r^{n+1}\left(\Pold-\pmb{P}^{n-1}\right)\chi_{S_n}:\testphi\,dxdt\right\rvert\\
        &\leq \Delta t_m\|\testphi\|_{L^\infty([0,T)\times\Omega)}\sum_{n=0}^{N-1}\int_\Omega |r^{n+1}|\left|\Pold-\pmb{P}^{n-1}\right| dx\\
        &\leq \Tilde{L}\Delta t_m\|\testphi\|_{L^\infty([0,T)\times\Omega)}\sum_{n=0}^{N-1}\int_\Omega |r^{n+1}|\left|\Qold-\pmb{Q}^{n-1}\right|dx\\
        &\leq \Tilde{L}\Delta t_m\left(\max_{0\leq n\leq N-1}\|r^{n+1}\|\right)\|\testphi\|_{L^\infty([0,T)\times\Omega)}\sum_{n=0}^{N-1}\|\Qold-\pmb{Q}^{n-1}\|\\
        &\leq  \Tilde{L}T^{\frac{1}{2}}\Delta t_m^{\frac{1}{2}}\left(\max_{0\leq n\leq N-1}\|r^{n+1}\|\right)\|\testphi\|_{L^\infty([0,T)\times\Omega)}\left(\sum_{n=0}^{N-1}\|\Qold-\pmb{Q}^{n-1}\|^2\right)^{\frac{1}{2}}\\
        &\leq C\Delta t_m^{\frac{1}{2}}\left(\max_{0\leq n\leq N-1}\|r^{n+1}\|\right)\|\testphi\|_{L^\infty([0,T)\times\Omega)}\left(\sum_{n=0}^{N-1}\|\nabla\Qold-\nabla\pmb{Q}^{n-1}\|^2\right)^{\frac{1}{2}}\to0.
    \end{aligned}
\end{equation*}
The estimate for $K_2$ is similar. Specifically, we have
\begin{equation*}
    \label{eq:K_2}
    \begin{aligned}
        \lvert K_2\rvert &=\Bigg\lvert \int_0^T\int_\Omega \rsolsubseq\left\{ \sum_{n=0}^{N-1}  \pmb{P}^{n-1}\,\chi_{S_n}-\sum_{n=0}^{N-1}  \left[\frac{(n+1)\Delta t_m-t}{\Delta t_m}\,\pmb{P}^{n-1}+\frac{t-n\Delta t_m}{\Delta t_m}\,\Pold\right]\,\chi_{S_n} \right\}:\testphi\,dxdt\Bigg\rvert\\
        &=\Bigg\lvert \int_0^T\int_\Omega \rsolsubseq \sum_{n=0}^{N-1}\left[  \frac{t-n\Delta t_m}{\Delta t_m}\,\left(\pmb{P}^{n-1}-\Pold\right)\chi_{S_n}\right]:\testphi\,dxdt\Bigg\rvert\\
        &\leq C\Delta t_m\|\rsolsubseq\|_{L^\infty([0,T);L^2(\Omega))}\|\testphi\|_{L^\infty([0,T)\times\Omega)}\sum_{n=0}^{N-1}\|\Qold-\pmb{Q}^{n-1}\|\\
        &\leq C\Delta t_m^{\frac{1}{2}}\|\rsolsubseq\|_{L^\infty([0,T);L^2(\Omega))}\|\testphi\|_{L^\infty([0,T)\times\Omega)}\left(\sum_{n=0}^{N-1}\|\nabla\Qold-\nabla\pmb{Q}^{n-1}\|^2\right)^{\frac{1}{2}}\to0.
    \end{aligned}
\end{equation*}
$K_3$ tends to $0$ as well thanks to the strong convergence of $\Psolsubseq$ to $P(\pmb{Q})$, see~\eqref{eq:PQsol_convergence},~\eqref{eq:Q-Q*}.  The last term $K_4$ goes to $0$ as $m$ tends to infinity by the weak convergence of $\rsolsubseq$ towards $r$. Thus, $\widetilde{r\pmb{P}}_{\Delta t_m}\rightharpoonup rP(\pmb{Q})$.

Using this, we prove $\Hsolsubseq\rightharpoonup\pmb{H}$. Indeed, we have
\begin{equation}
    \label{eq:H_convergence}
    \begin{aligned}
        &\int_0^T\int_\dom \Hsolsubseq:\testphi\,dxdt-\int_0^T\int_\dom \pmb{H}:\testphi\,dxdt\\
       &=-\int_0^T\int_\dom L(\nabla\Qsolsubseq-\nabla \pmb{Q}):\nabla\testphi-\int_0^T\int_\Omega \left(\widetilde{r\pmb{P}}_{\Delta t_m}-rP(\pmb{Q})\right):\testphi dxdt
        \to0,
    \end{aligned}
\end{equation}
as $m$ tends to infinity since $\Grad\Qsolsubseq\to \pmb{Q}$ in $L^2$. This shows that $\Hsolsubseq\rightharpoonup\pmb{H}$. 

Acording to \eqref{eq:SnSigman}, we can define
\begin{equation*}
    \label{eq:Ssol}
    \Ssol= \Wsol\,\Qbacksol-\Qbacksol\,\Wsol+\xi(\Qbacksol\,\Dsol+\Dsol\,\Qbacksol)+\frac{2\xi}{d}\Dsol-2\xi(\Dsol:\Qbacksol)\,(\Qbacksol+\frac{1}{d}\pmb{I}),
\end{equation*}
\begin{equation*}
    \label{eq:Sigmasol}
    \Sigmasol=\Qbacksol\,\Hsol-\Hsol\,\Qbacksol-\xi(\Hsol\,\Qbacksol+\Qbacksol\,\Hsol)-\frac{2\xi}{d}\Hsol+2\xi(\Qbacksol:\Hsol)\left(\Qbacksol+\frac{1}{d}\pmb{I}\right),
\end{equation*}
where 
\begin{equation*}
    \label{eq:DsolWsol}
    \pmb{D}_{\Delta t}^* = \frac{1}{2}(\nabla \uappr+\nabla\pmb{u}_{\Delta t}^{*\,\intercal}),\quad\quad   \pmb{W}_{\Delta  t}^*=\frac{1}{2}(\nabla\uappr-\nabla\pmb{u}_{\Delta t}^{*\,\intercal}).
\end{equation*}
Taking $\pmb{S}=S(\nabla\pmb{u},\pmb{Q})$, $\pmb{\Sigma}=\Sigma(\pmb{Q}, \pmb{H})$ and $\pmb{D}=\frac{1}{2}(\nabla\pmb{u}+\nabla\pmb{u}^\intercal)$, $\pmb{W}=\frac{1}{2}(\nabla\pmb{u}-\nabla\pmb{u}^\intercal)$, we claim that $\Ssolsubseq\rightharpoonup\pmb{S}$ and $\Sigmasolsubseq\rightharpoonup\pmb{\Sigma}$. Using formula \eqref{eq:numerical_s}, we can rewrite $\Ssolsubseq$ as 
\begin{equation*}
    \label{eq:s_subseq}
    \Ssolsubseq=S(\uapprsubseq, \Qbacksolsubseq)-\frac{2\xi}{d^2}\,(\Div \uapprsubseq)\pmb{I}.
\end{equation*}
 As it is shown in \eqref{eq:u_sol_convergence}, \eqref{eq:u-u*} and \eqref{eq:usol_divergence_free}, $\Div \uapprsubseq\rightharpoonup \Div \pmb{u}=0$, and so we only need to show  $S(\uapprsubseq, \Qbacksolsubseq)\rightharpoonup \pmb{S}$. The most challenging term to treat within $S(\uapprsubseq, \Qbacksolsubseq)$ is $(\Dsolsubseq:\Qbacksolsubseq)\,\Qbacksolsubseq$. The weak convergence of other terms follows in a similar way. Applying the generalized H\"{o}lder's inequality and Sobolev inequality, for any smooth function $\testvar$ with compact support in $[0,T)\times\Omega$, we obtain,
\begin{equation}
    \label{eq:S_convergence}
    \begin{aligned}
        &\left|\int_0^T\int_\dom (\Dsolsubseq:\Qbacksolsubseq)\,\Qbacksolsubseq:\testvar\,dxdt-\int_0^T\int_\dom (\pmb{D}:\pmb{Q})\,\pmb{Q}:\testvar\,dxdt\right|\\
        &=\Bigg|\int_0^T\int_\dom (\Dsolsubseq:\Qbacksolsubseq)\,(\Qbacksolsubseq-\pmb{Q}):\testvar\,dxdt+\int_0^T\int_\dom \left(\Dsolsubseq:(\Qbacksolsubseq-\pmb{Q})\right)\,\pmb{Q}:\testvar\,dxdt\\
        &\hphantom{=\Bigg|\int_0^T\int_\dom (\Dsolsubseq:\Qbacksolsubseq)\,(\Qbacksolsubseq-\pmb{Q}):\testvar\,dxdt+}
        +\int_0^T\int_\dom \left((\Dsolsubseq-\pmb{D}):\pmb{Q}\right)\,\pmb{Q}:\testvar\,dxdt\Bigg|\\
        &\leq C\|\testvar\|_{L^\infty([0,T)\times\dom)}\|\Dsolsubseq\|_{L^2([0,T)\times\dom)}\|\Qbacksolsubseq\|_{L^\infty(0,T; L^4)}\|\Qbacksolsubseq-\pmb{Q}\|_{L^2([0,T; L^4)}\\
        &\qquad+ C\|\testvar\|_{L^\infty([0,T)\times\dom)}\|\Dsolsubseq\|_{L^2([0,T)\times\dom)}\|\pmb{Q}\|_{L^\infty([0,T); L^4)}\|\Qbacksolsubseq-\pmb{Q}\|_{L^2([0,T); L^4)}\\
        &\qquad+ \left|\int_0^T\int_\dom\left((\Dsolsubseq-\pmb{D}):\pmb{Q}\right)\,\pmb{Q}:\testvar\,dxdt\right|\\
        &\leq C\left( \|\Qbacksolsubseq\|_{L^\infty(0,T; H^1)}+ \|\pmb{Q}\|_{L^\infty(0,T; H^1)}\right)\|\Qbacksolsubseq-\pmb{Q}\|_{L^2(0,T; H^1)}\\
        &\qquad+ \left|\int_0^T\int_\dom\left((\Dsolsubseq-\pmb{D}):\pmb{Q}\right)\,\pmb{Q}:\testvar\,dxdt\right|
    \end{aligned}
\end{equation}
As $m$ tends to infinity, the first term goes to $0$ since $\Qbacksolsubseq\to \pmb{Q}$ in $L^2(0, T; H^1)$. While the second term tends to $0$ because $\pmb{D}_{\Delta t_m}^*\rightharpoonup\pmb{D}$ in $L^2$,  and $(\pmb{Q}:\testvar)\,\pmb{Q}\in L^2([0, T)\times\dom)$.

To show $\Sigmasolsubseq\rightharpoonup\Sigma$ is similar, therefore, we will only present the treatment of the most challenging term within $\Sigmasolsubseq$, which is $(\Qbacksolsubseq:\Hsolsubseq)\,\Qbacksolsubseq$. For any smooth function $\testvar$ with compact support in $[0,T)\times\Omega$, we have
\begin{equation}
    \label{eq:Sigma_convergence}
    \begin{aligned}
        &\left|\int_0^T\int_\dom (\Qbacksolsubseq:\Hsolsubseq)\,\Qbacksolsubseq:\testvar\,dxdt-\int_0^T\int_\dom (\pmb{Q}:\pmb{H})\pmb{Q}:\testvar\,dxdt\right|\\
        &=\Bigg|\int_0^T\int_\dom (\Qbacksolsubseq:\Hsolsubseq)\,(\Qbacksolsubseq-\pmb{Q}):\testvar\,dxdt+\int_0^T\int_\dom \left((\Qbacksolsubseq-\pmb{Q}):\Hsolsubseq\right)\,\pmb{Q}:\testvar\,dxdt\\
        &\quad\quad\quad\quad\quad\quad\quad\quad\quad\quad\quad\quad\quad\quad\quad\quad\quad\quad\quad\quad+\int_0^T\int_\dom \left(\pmb{Q}:(\Hsolsubseq-\pmb{H})\right)\,\pmb{Q}:\testvar\,dxdt\Bigg|\\
        &\leq C\|\testvar\|_{L^\infty([0,T)\times\dom)}\left(\|\Qbacksolsubseq\|_{L^\infty(0,T; L^4)}+\|\pmb{Q}\|_{L^\infty(0,T; L^4)}\right)\|\Hsolsubseq\|_{L^2([0,T)\times\dom)}\|\Qbacksolsubseq-\pmb{Q}\|_{L^2(0,T; L^4)}\\
        &\qquad+\left|\int_0^T\int_\dom \left(\pmb{Q}:(\Hsolsubseq-\pmb{H})\right)\,\pmb{Q}:\testvar\,dxdt\right|\\
        &\leq C\|\Qbacksolsubseq-\pmb{Q}\|_{L^2(0,T; H^1)}+\left|\int_0^T\int_\dom \left(\pmb{Q}:(\Hsolsubseq-\pmb{H})\right)\,\pmb{Q}:\testvar\,dxdt\right|
    \end{aligned}
\end{equation}
        As $m$ tends to infinity, the first term goes to $0$ since $\Qbacksolsubseq\to \pmb{Q}$ in $L^2(0, T; H^1)$. The second term tends to $0$ because $\Hsolsubseq\rightharpoonup\pmb{H}$ in $L^2$, and $(\pmb{Q}:\testvar)\pmb{Q}\in L^2([0,T)\times\Omega)$.

\pmb{Step 5: Passing the limit.} Using the results from the previous steps, we can pass to the limit in most terms in weak formulation \eqref{eq:Dtu_step1}, and \eqref{eq:DtQ_step1} after integrating over $[0,T)$. The only two remaining terms remaining are $\int_0^T\int_\dom \Hsolsubseq\nabla\Qbacksolsubseq\cdot\pmb{\psi}\,dxdt$ and $\int_0^T\int_\dom (\uapprsubseq\cdot\nabla\Qbacksolsubseq):\testvar\,dxdt$. Combining weak and strong convergence as in Step 4, it follows that
\begin{equation}
    \label{eq:two_stretch_terms_convergence}
   \begin{split}
    &\int_0^T\int_\dom \Hsolsubseq\nabla\Qbacksolsubseq\cdot\pmb{\psi}\,dxdt\rightharpoonup \int_0^T\int_\dom \pmb{H}\nabla\pmb{Q}\cdot\pmb{\psi}\,dxdt,\\
    &\int_0^T\int_\dom (\uapprsubseq\cdot\nabla\Qbacksolsubseq):\testvar\,dxdt\rightharpoonup \int_0^T\int_\dom(\pmb{u}\cdot\nabla\pmb{Q}):\testvar\,dxdt.
   \end{split}
\end{equation}
This shows that $(\pmb{u}, \pmb{Q},\pmb{H}, r)$ is a weak solution satisfying Definition~\ref{def:weakreformulation}.
\end{proof}

In order to conclude, we need to show that the reformulated system~\eqref{eq:ut}-\eqref{eq:H} and the original hydrodynamics system~\eqref{eq:original_system_ut}-\eqref{eq:original_system_Qt} are equivalent in the weak sense. This follows from the following lemma that was proved in~\cite[Lemma 5.2]{GWY2020}:

\begin{lemma}
    \label{cor:equivalence_r}
    Assume that $(\pmb{u}, \pmb{Q} ,\pmb{H}, r)$ is a weak solution in the sense of Definition \eqref{def:weakreformulation}. Then for any smooth function $\phi$ with compact support in $(0,T)\times\Omega$ (compactly supported in both time and space), we have
    \begin{equation}
        \label{eq:r_equivalence}
        \int_0^T\int_\dom r\phi\,dxdt=\int_0^T\int_\dom r(\pmb{Q})\phi\,dxdt
    \end{equation}
   where $r(\pmb{Q})$ is defined in \eqref{eq:auxiliary_variable_r}.
\end{lemma}
\begin{proof}
    Since we have shown that $\partial_t \rsol\in L^2(0,T;L^1(\dom))$ in Corollary \ref{cor:Dtr_estimate}, the proof follows in the same way as \cite[Lemma 5.2]{GWY2020}.
\end{proof}
\begin{remark}
    \label{rem:density_L2_argument}
    Since the smooth functions are dense in $L^2(\Omega)$, it is also valid to choose $L^2$ functions as test functions in~\eqref{eq:r_equivalence} since $r$ is bounded in $L^2$.
\end{remark}

\bibliographystyle{abbrv}
\bibliography{related_literature}
\end{document}